\documentclass[a4paper,english]{article}
\usepackage[a4paper]{geometry}
\usepackage{amsthm,amsmath,amssymb,mathtools}
\usepackage{enumerate,enumitem,csquotes,verbatim}
\usepackage{setspace}
\usepackage[numbers]{natbib}
\usepackage{hyperref}
\usepackage{color}
\usepackage{tikz}

\theoremstyle{plain}
\newtheorem*{theorem*}{Theorem}
\newtheorem{theorem}{Theorem}
\newtheorem{lemma}[theorem]{Lemma}

\newtheorem*{claim*}{Claim}

\newtheorem{problem}[theorem]{Problem}

\theoremstyle{definition}

\theoremstyle{remark}

\DeclareMathOperator\In{in}

\title{An Isoperimetric Inequality and Pursuit-Evasion Games on Triangular Grid Graphs}

\author{Athipatana Iamphongsai\thanks{\,Triam Udom Suksa School, Bangkok 10330, Thailand; \texttt{athipatiamphongsai@gmail.com}.}
\and 
Teeradej Kittipassorn\thanks{\,Department of Mathematics and Computer Science, Faculty of Science, Chulalongkorn University, Bangkok 10330, Thailand; \texttt{teeradej.k@chula.ac.th}.}}

\begin{document}
\maketitle

%%%%%%%%%%%%%%%%%%%%
\begin{abstract}
In this paper, we prove an isoperimetric inequality for the triangular grid graph which was conjectured by Adams, Gibson, and Pfaffinger, using the compression technique of Bollob\'as and Leader. Moreover, we apply the isoperimetric inequality to the Zero-Visibility Search game and Lions and Contamination game in order to obtain lower bounds for the inspection number and lion number respectively. We also provide searching strategies to prove upper bounds for both the inspection number and lion number. 
\end{abstract}

%%%%%%%%%%%%%%%%%%%%
\section{Introduction}

Isoperimetric inequalities have long intrigued us. One of the earliest problems of this matter considered by the Greeks was to find the maximum area of a region given its perimeter. The solution is that the circle is best for maximizing the area. This problem can also be considered in a discrete setting, specifically on graphs. Given a graph $G$, we view a vertex subset $S$ of $G$ as a closed space with its area being the number of vertices in $S$ and its perimeter being the \emph{vertex boundary} $\partial_{G} (S)$, the collection of all vertices not in $S$ that share an edge with some vertex in~$S$.

Discrete isoperimetric inequalities have been studied in many different settings. Bollob\'as and Leader~\cite{Torus,Compression} studied discrete isoperimetric inequalities on the grid graphs and the discrete toruses. The hypercubes were also considered by Bobkov~\cite{Discretecube}, while Bharadwaj and Chandran~\cite{TreeIntro} studied the inequalities on trees. See also~\cite{Chung,Sphere,Wang} for other treatments of discrete isoperimetric inqualities. Common methods used in solving isoperimetric problems have been proposed including  Azuma's inequality~\cite{Azuma}, the concentration of measure method~\cite{Concentration}, and the compression technique~\cite{Compression}.
% Discrete isoperimetric inequalities have a relatively short history, beginning in the 1960's. However, since the introduction of these problems many researchers~\cite{TreeIntro,Torus,Chung,Wang} have studied these problems and have proposed methods for approaching these sort of problems. Common methods used in solving isoperimetric problems are the concentration of measure method~\cite{Concentration}, Azuma's inequality~\cite{Azuma} and the compression technique~\cite{Compression}.
 
 %Works on the topic of discrete isoperimetric inequalities include Bharadwaj and Chandran~\cite{TreeIntro} who studied discrete isoperimetric inequalities on trees and Bollob\'as and Leader~\cite{Torus} who studied discrete isoperimetic inequalities on discrete toruses. See also~\cite{Chung,Wang} for other works on this topic.\\

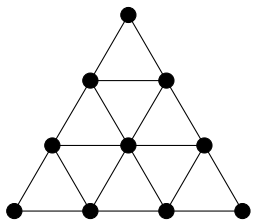
\begin{figure}[h]
	\centering
	\begin{tikzpicture}
		\filldraw (0,0) circle(0.1cm);
		\filldraw (1,0) circle(0.1cm);
		\filldraw (2,0) circle(0.1cm);
		\filldraw (3,0) circle(0.1cm);
		\filldraw (0.5,0.87) circle(0.1cm);
		\filldraw (1.5,0.87) circle(0.1cm);
		\filldraw (2.5,0.87) circle(0.1cm);
		\filldraw (1,1.73) circle(0.1cm);
		\filldraw (2,1.73) circle(0.1cm);
		\filldraw (1.5,2.60) circle(0.1cm);
		
		\draw (0,0) -- (3,0);
		\draw (0,0) -- (0.5,0.87);
		\draw (1,0) -- (0.5,0.87);
		\draw (2,0) -- (1.5,0.87);
		\draw (1,0) -- (1.5,0.87);
		\draw (2,0) -- (2.5,0.87);
		\draw (3,0) -- (2.5,0.87);
  	    \draw (0.5,0.87) -- (2.5,0.87);
   	    \draw (2,1.73) -- (1,1.73);	
            \draw (0.5,0.87) -- (1,1.73);
            \draw (1.5,0.87) -- (1,1.73);
            \draw (1.5,0.87) -- (2,1.73);
            \draw (2.5,0.87) -- (2,1.73);
            \draw (1.5,2.60) -- (2,1.73);
            \draw (1.5,2.60) -- (1,1.73);

	\end{tikzpicture}
	\caption{A triangular grid graph $T_3$}
	\label{fig:triangulargraph}
\end{figure}
% \begin{figure}
%     \centering
%     \includegraphics{T_4.png}
%     \caption{The triangular grid graph $T_4$}
%     \label{fig:1}
% \end{figure}

In this paper, we focus on a discrete isoperimetric inequality on the \emph{triangular grid graph} denoted by $T_n$. More precisely, the graph $T_n$ is the $(n-1)$-th subdivision of a triangle containing $\frac{(n+1)(n+2)}{2}$ vertices as shown in Figure~\ref{fig:triangulargraph}. While studying a pursuit-evasion game on $T_n$, Adams, Gibson, and Pfaffinger~\cite{Lions} needed a discrete isoperimetric inequality on $T_n$. They made a conjecture which states that if $A$ is a vertex subset of $T_n$ with a given size, then $|\partial_{T_n} (A)|$ is minimized when $A$ is a \emph{row packing} or an \emph{ice cream cone packing}. We define a \emph{row packing} as a vertex packing where each row of $T_n$ can only be filled once the row below it is already full, and within each row the vertices are filled from left to right. A vertex subset is an \emph{ice cream cone packing} if its complement is a row packing. We prove the conjecture.
\begin{theorem}\label{thm:inequality}
The boundary of a vertex subset of $T_n$ of size $k$ is at least as large as the minimum of that of the row packing and ice cream cone packing of size $k$.
\end{theorem}

The proof is based on the compression technique of Bollob\'as and Leader~\cite{Compression}. We will demonstrate some applications of this discrete isoperimetric inequality on pursuit-evasion games.

Graph-based pursuit-evasion games can be seen as games where the objective is for a searcher to capture an intruder who resides within the graph and can move in some manner within the graph. Depending on the conditions of the graph, searcher, intruder and movement etc., we obtain different variants of pursuit-evasion games. Pursuit-evasion games have many applications in air traffic control~\cite{traffic}, collision avoidance~\cite{collision}, and tracking~\cite{tracking}.
One of the most prominent pursuit-evasion games on graphs is the \emph{Cops and Robber} game which originated from the works of Aigner and Fromme~\cite{aigner}, Nowakowski and Winkler~\cite{nowakowski}, and Quilliot~\cite{quilliot}. In this particular pursuit-evasion game a set of cops moving along the edges of a graph attempt to capture a single robber who can move along the edges of a graph as well. During the game, the cops and robber alternate turns moving. The game ends when the robber is caught.

In the classic Cops and Robber game, cops know where the robber is within the graph. However, in some variants of pursuit-evasion games the searcher does not have any knowledge of where the intruder is residing in. These games are classified as \emph{zero visibility} search games. The \emph{Zero Visibility Cops and Robber} game was originally proposed by To\v{s}i\'{c}~\cite{Early}. The rules in this game are almost the same as the Cops and Robber game, except that the cops have no knowledge of where the robber is. Dereniowski, Dyer, Tifenbach, and Yang~\cite{zerocop2} investigated this model on trees and Neufeld and Nowakowski~\cite{Neufeld} considered this model on products of graphs. Another zero visibility search game is the \emph{Hunters} \& \emph{Rabbit} game which was developed from the works of Britnell and Wildon~\cite{hunter2}, and Haslegrave~\cite{hunter1}. In this game the searcher can choose to examine an arbitrary set of vertices within a graph, while the intruder must move along the edges of the graph in each turn. The searcher and intruder alternate turns, and the searcher has no knowledge of where the intruder is. The game ends when the intruder is caught. Abramovskaya, Fomin, Golovach, and Pilipczuk~\cite{hunter3} investigated the game on an $(m \times n)$-grid and trees, Gruslys and Me\'roueh~\cite{hunter4} studied this game in terms of cats and a mouse. Other studies on pursuit-evasion games, not limited to zero visibility games, can also be found in~\cite{Adler,Brass,Fomin,Boting}.
 
 % An example of the zero visibility search game is the Zero Visibility Cops and Robbers game studied by Dereniowski, Dyer and Yang~\cite{zero1}; Xue, Yang and Zilles~\cite{zero2}; Wang and Zhong~\cite{zero3}

In this paper, we will consider two variations of pursuit-evasion games.

The first variant of pursuit-evasion games we will consider is called the \emph{Zero-Visibility $k$-Search} game proposed by Bernshteyn and Lee~\cite{Search}. The rules of this game are extremely similar to that of the Hunters \& Rabbit game except that the intruder does not necessarily have to move during each turn. In this game, the searcher is allowed to examine any $k$ vertices during each turn which is equivalent to having $k$ hunters in the Hunters \& Rabbit game. A winning strategy for the searcher is a finite sequence of moves which guarantees that the intruder will always be caught no matter how the intruder moves along the graph. The minimum $k$ such that a winning strategy for the searcher exists is called the \emph{inspection number} of a graph $G$ and is denoted by $\In(G)$.
% It is a two-player game played on a graph $G$, where one player is the searcher while the other is the intruder. The intruder starts the game at any arbitrary vertex and in each turn can move to any adjacent vertex or stay in the same vertex he/she is residing in. On the other hand, when the searcher's turn comes by the searcher is allowed to examine any $k$ vertices in the graph. It is also important to note that the searcher does not know if and where the intruder moves. The game ends when the searcher examines the vertex that the intruder is residing in during one of the searcher's turns. 

We provide lower and upper bounds for the inspection number of the triangular grid graphs. We prove the lower bound by applying Theorem~\ref{thm:inequality} together with a result of Bernshteyn and Lee~\cite{Search}, and prove the upper bound $\lceil\frac{3n}{4}\rceil+2$.

\begin{theorem}\label{thm:search}
 For all positive integers $n$, we have $\frac{n}{\sqrt{2}}<\In(T_n)\le \lceil\frac{3n}{4}\rceil+2$.
\end{theorem}

The second variant of pursuit-evasion games we will consider is a version of the Zero Visibility Cops and Robber game where the cops and robber move simultaneously during each round. Moreover, if the robber transverses the same edge of the graph as a cop, the robber will be caught. This problem can be viewed from another approach proposed by Berger, Gilbers, Gr\"{u}ne and Klein~\cite{Berger} that makes the robber unnecessary. Instead of considering the single robber, they considered the set of all possible positions of the robber, and let these positions be represented by contamination. Also, they used lions as the searchers instead of cops. Since cops have no knowledge of where the intruder is residing in the Zero Visibility Cops and Robber game, the cops must eliminate all the possible positions of the robber to make sure that they are able to catch the robber. This is equivalent to the lions having to clear all of the contamination. 

Berger, Gilbers, Gr\"{u}ne and Klein called this game the \emph{Lions and Contamination} problem. At the start of the game, every vertex that is not occupied by a lion is contaminated. During each turn, the lions and contamination move simultaneously but the contamination cannot spread through an edge that the lions use to travel during a certain turn. The lions can choose to stay where they are, or move to an adjacent vertex. The contamination will travel along the edges of the graph to a vertex that is not blocked by a lion and can recontaminate previously cleared vertices. For example, if a lion moves out of a vertex using a certain edge, the contamination can spread into that vertex from a different edge in the same turn. A winning strategy for the lions is a finite sequence of moves which guarantees that the lions will clear the graph of contamination. The minimum number of lions $l$ such that a winning strategy exists is called the \emph{lion number} and is denoted by $l(G)$. 
% This game is precisely the game considered by Berger, Gilbers, Gr\"{u}ne and Klein~\cite{Berger}, and Bertschinger, Reddy, and Mann~\cite{monotonelion}. We will call this game the \emph{Lions and Contaminatiion} problem.

We will prove a lower bound for the lion number of triangular grid graphs conjectured by Adams, Gibson, and Pfaffinger~\cite{Lions} by illustrating a connection between the Lions and Contamination game and the Zero-Visibility Search game. Note that Adams, Gibson, and Pfaffinger~\cite{Lions} sketched another method to deduce the lower bound from Theorem~\ref{thm:inequality}. We also prove the upper bound using an idea similar to that of Adams, Gibson, and Pfaffinger~\cite{Lions}.

\begin{theorem}\label{thm:lion}
For all positive integers $n$, we have $\frac{n}{2\sqrt{2}}<l(T_n)\le n+1$.
\end{theorem}

The rest of this paper is organized as follows. Section~\ref{sec:thm:inequality} is devoted to the proof of Theorem~\ref{thm:inequality}. In Section~\ref{sec:thm:search}, we study the Zero Visibility $k$-Search game on the triangular grid graph, and prove Theorem~\ref{thm:search}. In Section~\ref{sec:thm:lion}, we study the Lions and Contamination game, and prove Theorem~\ref{thm:lion}. Finally, we conclude the paper in Section~\ref{sec:conclusion} with some open problems.

%%%%%%%%%%%%%%%%%%%%
\section{Proof of Theorem~\ref{thm:inequality}}
\label{sec:thm:inequality}

In this section we will prove a discrete isoperimetric inequality on the triangular grid graph $T_n$.

First, let us give a few definitions for the sake of the proof. Let a vertex $v$ of the triangular grid graph have coordinates $(v_1,v_2)$ where $v_1$ is the number of vertices within the same row that are to the left of $v$ and $v_2$ is the number of rows below the row of $v$. For a vertex $v$ of the triangular grid graph with coordinates $(v_1,v_2)$, define $|v|$ as the sum of $v_1$ and $v_2$. Also, let us shift the triangular grid graph, so that the vertices who share the same first coordinate are arranged in columns, and the vertices who share the same second coordinates are arranged in rows (see Figure~\ref{fig:shiftedtriangulargraph}).

\begin{figure}[h]
	\centering
	\begin{tikzpicture}
		\filldraw (0,0) circle(0.1cm);
		\filldraw (1,0) circle(0.1cm);
		\filldraw (2,0) circle(0.1cm);
		\filldraw (3,0) circle(0.1cm);
		\filldraw (0,1) circle(0.1cm);
		\filldraw (0,2) circle(0.1cm);
		\filldraw (0,3) circle(0.1cm);
		\filldraw (1,1) circle(0.1cm);
		\filldraw (1,2) circle(0.1cm);
		\filldraw (2,1) circle(0.1cm);
		
		\draw (0,0) -- (3,0);
		\draw (0,0) -- (0,3);
		\draw (1,0) -- (1,2);
		\draw (2,0) -- (2,1);
		\draw (0,3) -- (3,0);
		\draw (2,0) -- (0,2);
		\draw (1,0) -- (0,1);
  	    \draw (0,1) -- (2,1);
   	    \draw (0,2) -- (1,2);

            \node at (0.1,-0.5) {(0,0)};
            \node at (1.1,-0.5) {(1,0)};
            \node at (2.1,-0.5) {(2,0)};
            \node at (3.1,-0.5) {(3,0)};
            \node at (0.1,0.5) {(0,1)};
            \node at (1.1,0.5) {(1,1)};
            \node at (2.1,0.5) {(2,1)};
            \node at (0.1,1.5) {(0,2)};
            \node at (1.1,1.5) {(1,2)};
            \node at (0.1,2.5) {(0,3)};

	\end{tikzpicture}
	\caption{A shifted $T_3$}
	\label{fig:shiftedtriangulargraph}
\end{figure}
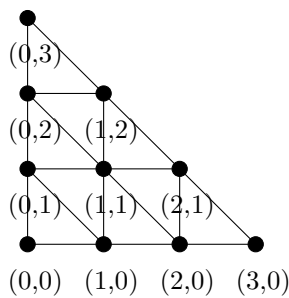

We say that \emph{$u$ comes before $v$ in the simplicial ordering} on the triangular grid graph if either $|u| < |v|$, or $|u| = |v|$ and $u_1 > v_1$.

% Let $A \subseteq T_n$. For $i = 1$, the \emph{$i$-projections} of $A$ are the sets $A_t$ defined as the set of projection points from column $t$ of $T_n$ onto column $1$ of $T_n$. We will also define $A_{t^{-}}$ as $A_t$ but with every projection point moved one coordinate down except the projection of $(t,0)$ which remains the same. On the other hand, $A_{t^{+}}$ is defined as $A_t$ but with every projection point moved one coordinate up except the projection of $(0,n)$ which remains the same.
% Similarly for $i = 2$, the $i$-projections of $A$ are the sets $A_t$ defined as the set of projection points from row $t$ of $T_n$ onto row $1$ of $T_n$. We will also define $A_{t^{-}}$ as $A_t$ but with every projection point moved one coordinate to the left except the projection of $(0,t)$ which remains the same. On the other hand, $A_{t^{+}}$ is defined as $A_t$ but with every projection point moved one coordinate to the right except the projection of $(n,0)$ which remains the same. 
% Similarly for $i = 1$, the $i$-projections of $A$ are the sets $A_t$ defined as the set of projection points from column $t$ of $T_n$ onto column $1$ of $T_n$. We will also define $A_{t^{-}}$ as $A_t$ but with every projection point moved one coordinate down except the projection of $(t,0)$ which remains the same. On the other hand, $A_{t^{+}}$ is defined as $A_t$ but with every projection point moved one coordinate up except the projection of $(0,n)$ which remains the same.
Let $A \subseteq V(T_n)$. The \emph{$1$-sections} of $A$ are the sets $A_{1,t} = \{ v_2 : (v_1,v_2) \in A , v_1=t \}$ for $0 \le t \le n$. Similarly, the \emph{$2$-sections} of $A$ are the sets $A_{2,t} = \{ v_1 : (v_1,v_2) \in A , v_2=t \}$ for $0 \le t \le n$. 
% For $S \subseteq \{0,1,2,\dots,n\}$, we define $S^{-} = \{x-1 \in \{0,1,2,\dots,n\}: x \in S\}$ and $S^{+} = \{x+1 \in \{0,1,2,\dots,n\}: x \in S\}$.

We let the \emph{$1$-left-compression} $C_1^{-}(A)$ of $A$ be obtained from $A$ by pushing the vertices in each column downward, and we let the \emph{$2$-left-compression} $C_2^{-}(A)$ of $A$ be obtained from $A$ by pushing the vertices in each row to the left as shown in Figure~\ref{fig:compressionimage}. Similarly, the \emph{$1$-right-compression} 
$C_1^{+}(A)$ of $A$ is obtained from $A$ by pushing the vertices in each column upward, and the 
\emph{$2$-right-compression} $C_2^{+}(A)$ of $A$ is obtained from $A$ by pushing the vertices in each row to the right.
More precisely for $i = 1,2$, $C_i^{-}(A)$ is defined by its $i$-sections as follows
\[(C_i^{-}(A))_{i,t} = \{0,1,2,\dots,|A_{i,t}|-1\},\]
and $C_i^{+}(A)$ is defined by its $i$-sections as follows
\[(C_i^{+}(A))_{i,t} = \{n-t,n-t-1,n-t-2,\dots,n+1-t-|A_{i,t}|\}.\]

\begin{figure}\label{compressionimage}
	\centering
	\begin{tikzpicture}
		\draw (0,0) circle(0.1cm);
		\draw (0,1) circle(0.1cm);
		\draw (0,2) circle(0.1cm);
		\draw (1,0) circle(0.1cm);
		\draw[fill = green] (1,1) circle(0.1cm);
		\draw[fill = green] (2,0) circle(0.1cm);

		\draw (0,0) -- (2,0);
		\draw (0,0) -- (0,2);
		\draw (1,0) -- (1,1);
		\draw (0,1) -- (1,1);
            \draw (0,1) -- (1,0);
            \draw (0,2) -- (2,0);

             \draw (3,0) circle(0.1cm);
		\draw (3,1) circle(0.1cm);
		\draw (3,2) circle(0.1cm);
		\draw[fill = green] (4,0) circle(0.1cm);
		\draw (4,1) circle(0.1cm);
		\draw[fill = green] (5,0) circle(0.1cm);
		
		\draw (3,0) -- (5,0);
		\draw (3,0) -- (3,2);
		\draw (4,0) -- (4,1);
		\draw (3,1) -- (4,1);
            \draw (3,1) -- (4,0);
            \draw (3,2) -- (5,0);

            \draw[fill = green] (6,0) circle(0.1cm);
		\draw[fill = green] (6,1) circle(0.1cm);
		\draw (6,2) circle(0.1cm);
		\draw (7,0) circle(0.1cm);
		\draw (7,1) circle(0.1cm);
		\draw (8,0) circle(0.1cm);
		
		\draw (6,0) -- (8,0);
		\draw (6,0) -- (6,2);
		\draw (7,0) -- (7,1);
		\draw (6,1) -- (7,1);
            \draw (6,1) -- (7,0);
            \draw (6,2) -- (8,0);

        \end{tikzpicture}

        \caption{The first image is a possible vertex subset of $T_3$ which we shall call $A$. The second image is $C_1^{-}(A)$. The third image is $C_2^{-}(A)$.}
	\label{fig:compressionimage}
\end{figure}
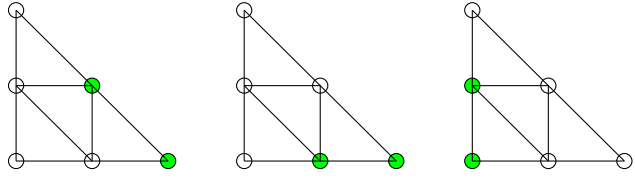

Next we will provide some lemmas necessary for the proof in both cases of Theorem~\ref{thm:inequality}. Lemma~\ref{neighborhood} shows that when a set of vertices is compressed, the neighborhood will never become larger. The \emph{neighborhood} $N(A)$ of a vertex subset $A$ of $T_n$ refers to the vertex subset $A \cup \partial_{T_n}(A)$. Lemma~\ref{emptyfulldiagonal} proves that the initial segment of the simplicial ordering minimizes the boundary size of vertex subsets that do not intersect with the diagonal $|x| = n$. Moreover, Lemma~\ref{emptyfulldiagonal} also proves that the final segment of simplicial ordering minimizes the boundary size of vertex subsets where all the vertices on the diagonal $|x| = n$ are members of that vertex subset.

\begin{lemma}\label{neighborhood}
If $A \subseteq V(T_n)$, then $|N(A)| \ge |N(C_i^{-}(A))|$ and $|N(A)| \ge |N(C_i^{+}(A))|$ for $i = 1,2$.
\end{lemma}

\begin{proof}
First, we will show that $|N(A)| \ge |N(C_i^{-}(A))|$. Fix $1 \le i \le 2$. Let $B = C_i^{-}(A)$ and $0 \le t \le n$.
We write $[j] = \{1,2,3,\dots,j\}$.
For convenience, we write $A_t$ for $A_{i,t}$ and similarly for other subsets of $V(T_n)$.
For $S \subseteq \{0,1,2,\dots,n\}$, we define $S^{-} = \{x-1 \in \{0,1,2,\dots,n\}: x \in S\}$ and $S^{+} = \{x+1 \in \{0,1,2,\dots,n\}: x \in S\}$.
Clearly
\[N(A)_t = A_t \cup A_t^{-} \cup (A_t^{+} \cap [n-t]) \cup A_{t-1} \cup A_{t-1}^{-} \cup A_{t+1} \cup A_{t+1}^{+}\]
and
\[N(B)_t = B_t \cup B_t^{-} \cup (B_t^{+} \cap [n-t]) \cup B_{t-1} \cup B_{t-1}^{-} \cup B_{t+1} \cup B_{t+1}^{+}.\]
% Because $B$ is a leftcompression of $A$, we deduce that 
Since $B_t \cup B_t^{-} \cup (B_t^{+} \cap [n-t]) = B_t \cup (B_t^{+} \cap [n-t])$,
$B_{t-1} \cup B_{t-1}^{-} = B_{t-1}$, and $B_{t+1} \cup B_{t+1}^{+}$ are nested, we deduce that
\[|B_t \cup B_t^{-} \cup B_t^{+} \cup B_{t-1} \cup B_{t-1}^{-} \cup B_{t+1} \cup B_{t+1}^{+}| \le \max\left\{|B_t \cup (B_t^{+} \cap [n-t])|, |B_{t-1}|, |B_{t+1} \cup B_{t+1}^{+}|\right\}.\] 
Also $|B_t \cup (B_t^{+} \cap [n-t])| \le |A_t \cup A_t^{-} \cup (A_t^{+} \cap [n-t])|$, $ |B_{t-1}| = |A_{t-1}|$ and $|B_{t+1} \cup B_{t+1}^{+}| \le |A_{t+1} \cup A_{t+1}^{+}|$.
Therefore, 
\[|N(B)_t| \le \max\left\{|A_t \cup A_t^{-} \cup (A_t^{+} \cap [n-t])|, |A_{t-1}|, |A_{t+1} \cup A_{t+1}^{+}|\right\} \le |N(A)_t|.\] 
Thus,
\[|N(A)| = \sum_{t=0}^n |N(A)_t| \ge \sum_{t=0}^n |N(B)_t| = |N(B)| = |N(C_i^{-}(A))|.\]

For the inequality $|N(A)| \ge |N(C_i^{+}(A)|$, view the compressions in the form of Figure~\ref{fig:shiftedtriangulargraph} in the form of Figure~\ref{fig:triangulargraph} by shifting all the vertices and edges from a triangular grid graph in the form of Figure~\ref{fig:shiftedtriangulargraph} so that it is in the form Figure~\ref{fig:triangulargraph}. 
Observe that the $2$-right-compression $C_2^{+}(A)$ is the reflection of the $2$-left-compression $C_2^{-}(A)$ about the height from the top vertex.
Thus, $|N(A)| \ge |N(C_2^{+}(A))|$ follows from $|N(A)| \ge |N(C_2^{-}(A))|$. Similarly, $|N(A)| \ge |N(C_1^{+}(A))|$ \qedhere
\end{proof}

% \begin{lemma}\label{neighborhood right}
% If $A \subseteq T_n$, then $|N(A)| \ge |N(C_i(A)^{+})|$ for all $i$.
% \end{lemma}

% \begin{proof}[Proof of Lemma~\ref{neighborhood right}.]
% When a triangular grid graph is viewed in the same form as that of Figure~\ref{fig:triangulargraph}, 
% it can be seen that the $i$-rightcompression is symmetrical to that of the $i$-leftcompression because the $i$-rightcompression can be generated by rotating the $i$-leftcompression.
% Thus, Lemma~\ref{neighborhood right} follows from Lemma~\ref{neighborhood left}. \qedhere
% \end{proof}

\begin{lemma}
\begin{enumerate}

\item If $A \subseteq V(T_n)$ and the diagonal $|x| = n$ has no member of $A$, then $|N(A)| \ge |N(C)|$ where $C$ is the initial segment of the simplicial ordering of size $|A|$.

\item If $A \subseteq V(T_n)$ and all the vertices on the diagonal $|x| = n$ are members of $A$, then $|N(A)| \ge |N(D)|$ where $D$ is the final segment of the simplicial ordering of size $|A|$.

\end{enumerate}
\label{emptyfulldiagonal}
\end{lemma}

\begin{proof}
% From Lemma~\ref{neighborhood left}, we get that $|N(C_i(A)^{-})| \le |N(A)|$. By definition of the leftcompression of $A$, we can also conclude that $|C_i(A)^{-}| = |A|$ and the diagonal $|x| = n$ has no member of $C_i(A)^{-}$. 
First, we will prove ($i$). Among all $B \subseteq V(T_n)$ satisfying 
$|B| = |A|, |N(B)| \le |N(A)|$ and the diagonal $|x| = n$ has no member of $B$, choose $B$ such that the sum of all positions in the simplicial ordering of all the vertices in $B$ is minimum. We must get that $B$ is $i$-left-compressed for all $i$, otherwise $C_i^{-}(B)$ would satisfy all of the previous conditions by Lemma~\ref{neighborhood} and would also contradict the minimality of $B$. 
It suffices to show that $|N(B)| \ge |N(C)|$.

Let $r = \min\{|x| : x \notin B\}$ and 
$s = \max\{|x| : x \in B\}$.
If $r > s$, then $B = \{x : |x| \le s\} = C$ which implies that $|N(B)| \ge |N(C)|$. If $r = s$, then $\left\{x : |x| < r\right\} \subset B \subset \left\{x : |x| \le r\right\}$, so certainly $|N(B)| \ge |N(C)|$. 

Assume $r < s$ as shown in Figure~\ref{fig: r<s diagram}. Since $B$ is $i$-left-compressed for all $i$, then $\{x : |x| = s\} \not\subset B$, there exists vertices $x \in B$ and $x{'} \notin B$ such that $|x| = |x{'}| = s$ and $x$, $x{'}$ are adjacent. Likewise, there exists vertices $y \in B$ and $y^{'} \notin B$ such that $|y| = |y{'}| = r$ and $y$, $y{'}$ are adjacent. Set $B{'} = B \cup 
\{y{'}\} \setminus \{x\}$. It can be seen that $|N(B{'})| \le |N(B)|$ since removing $x$ from $B$ decreases the neighborhood size by at least $1$, while adding $y{'}$ into $B$ increases the neighborhood size by at most $1$. Indeed, the common neighbor of $x$ and $x{'}$ in the diagonal $s+1$ will certainly be removed from the neighborhood of $B$ since $s+1 \le n$. On the other hand, the neighbor of $y{'}$ on the diagonal $r+1$ which is not adjacent to $y$ is the only vertex that may be added into the original $N(B)$. Evidently, the diagonal $|x| = n$ does not have a member of $B{'}$ and the sum of all the positions in the simplicial ordering of all the vertices in $B{'}$ is less than that of $B$. This is a contradiction with the minimality of $B$. Hence, the case $r < s$ cannot happen.
Thus, $|N(A)| \ge |N(B{'})| \ge |N(C)|$.

% Assume that 
% \[1+2+3+\dots+l < |A| = |C| \le 1+2+3+\dots+(l+1).\]
% Since $C$ is the initial segment of the simplicial ordering, we get that all columns from column $1$ to column $l+1$ must all contain a vertex that is a member of $C$. Moreover, the diagonal $|x| = n$ has no member of $C$. Thus, the vertex on top of the highest member of $C$ in each column must be in the vertex boundary of $C$. This implies that column $1$ to column $l+1$ has at least $l+1$ vertices in the vertex boundary of $C$. Moreover, the vertex that is to the right of the rightmost member of $C$ in row $1$ must also be in the vertex boundary of $C$. Thus, $|\partial_{T_n}(C)| \ge l+2$. From $|N(A)| \ge |N(C)|$ and $|A| = |C|$, we conclude that $|\partial_{T_n}(A)| \ge |\partial_{T_n}(C)| = l+2$.

Next, we will prove ($ii$). Among all $B \subseteq V(T_n)$ satisfying 
$|B| = |A|, |N(B)| \le |N(A)|$ and all vertices on the diagonal $|x| = n$ are members of $B$, choose $B$ such that the sum of all positions in the simplicial ordering of all the vertices in $B$ is maximum. We must get that $B$ is $i$-right-compressed for all $i$, otherwise $C_i^{+}(B)$ would satisfy all of the previous conditions by
Lemma~\ref{neighborhood} and would also contradict the maximality of $B$. 
It suffices to show that $|N(B)| \ge |N(D)|$.

Let $r =$ max$\{|x| : x \notin B\}$ and $s =$ min$\{|x| : x \in B\}$. If $r < s$, then $B = \left\{x : |x| \ge r\right\} = D$ which implies that $|N(B)| \ge |N(D)|$. If $r = s$, then $\{x : |x| > r\} \subset B \subset \{x : |x| \ge r\}$, so certainly $|N(B)| \ge |N(D)|$.

Assume $r > s$ as shown in Figure~\ref{fig: r>s diagram}. Since $B$ is $i$-right-compressed for all $i$, we get that $\{x : |x| = s\} \not\subset B$. Thus, there exists vertices $x \in B$ and $x{'} \notin B$ such that $|x| = |x{'}| = s$ and $x$, $x{'}$ are adjacent. Likewise, there exists vertices 
$y \in B$ and $y{'} \notin B$ such that 
$|y| = |y{'}| = r$ and $y$, $y{'}$ are adjacent.
Set $B{'} = B \cup \{y{'}\} \setminus \{x\}$.
It can be seen that $|N(B{'})| \le |N(B)|$ since removing $x$ from $B$ decreases the neighborhood size by at least 1, while adding $y{'}$ into $B$ increases the neighborhood size by at most 1. Indeed, the common neighbor of $x$ and $x{'}$ on the diagonal $s-1$ will certainly be removed from the neighborhood of $B$. On the other hand, the neighbor of $y{'}$ on the diagonal $r-1$ which is not adjacent to $y$ is the only vertex that may be added into the original $N(B)$ since $r-1 \le n-2$.
Evidently, all the vertices on the diagonal $|x| = n$ are members of $B{'}$ and the sum of all the positions in the simplicial ordering of all the vertices in $B{'}$ is greater than that of $B$. This is a contradiction with the maximality of $B$. Hence, the case $r > s$ cannot happen. Thus, $|N(A)| \ge |N(B)| \ge |N(D)|$. \qedhere

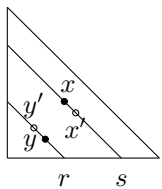
\begin{figure}[h]
    \centering
    \begin{tikzpicture}
        
        \draw (0,0) -- (0,2);
        \draw (0,0) -- (2,0);
        \draw (0,2) -- (2,0);
        \draw (0.75,0) -- (0,0.75);
        \draw (1.5,0) -- (0,1.5);

        \node at (0.75,-0.3) {$r$};
        \node at (1.5,-0.3) {$s$};

        \filldraw (0.5,0.25) circle(0.04cm);
        \draw (0.35,0.40) circle(0.04cm);
        \filldraw (0.75,0.75) circle(0.04cm);
        \draw (0.9,0.6) circle(0.04cm);

        \node at (0.30,0.20) {$y$};
        \node at (0.35,0.70) {$y'$};
        \node at (0.8,0.95) {$x$};
        \node at (0.9,0.4) {$x'$};

    \end{tikzpicture}
    \caption{A diagram of the case $r<s$ for ($i$).}
    \label{fig: r<s diagram}
\end{figure}

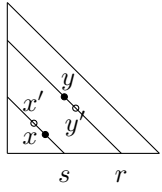
\begin{figure}[h]
    \centering
    \begin{tikzpicture}
        
        \draw (0,0) -- (0,2);
        \draw (0,0) -- (2,0);
        \draw (0,2) -- (2,0);
        \draw (0.75,0) -- (0,0.75);
        \draw (1.5,0) -- (0,1.5);

        \node at (0.75,-0.3) {$s$};
        \node at (1.5,-0.3) {$r$};

        \filldraw (0.5,0.25) circle(0.04cm);
        \draw (0.35,0.40) circle(0.04cm);
        \filldraw (0.75,0.75) circle(0.04cm);
        \draw (0.9,0.6) circle(0.04cm);

        \node at (0.30,0.20) {$x$};
        \node at (0.35,0.70) {$x'$};
        \node at (0.8,0.95) {$y$};
        \node at (0.9,0.4) {$y'$};

    \end{tikzpicture}
    \caption{A diagram of the case $r>s$ for ($ii$).}
    \label{fig: r>s diagram}
\end{figure}

% Assume that 
% \[(n+1)+n+(n-1)+\dots+(l+1) \le |A| = |D| < (n+1)+n+(n-1)+\dots+l.\] 
% Since $D$ is the final segment of the simplicial ordering, the vertices in $D$ must occupy column $1$ to column $l$ and each column must have a vertex that is not a member of $D$. This implies that every column from column $1$ to column $l$ must have at least 1 boundary vertex of $D$. Therefore, $|\partial_{T_n}(D)| \ge l$. From $|N(A)| \ge |N(D)|$ and $|A| = |D|$, we conclude that $|\partial_{T_n}(A)| \ge |\partial_{T_n}(D)| \ge l$.   \qedhere
\end{proof}

\begin{lemma}
\begin{enumerate}

\item If 
\[1+2+3+\dots+l < |C| \le 1+2+3+\dots+(l+1)\] where $C$ is the intitial segment of the simplicial ordering,
then $|\partial_{T_n}(C)| = l+2$.
\item If 
\[(n+1)+n+(n-1)+\dots+(l+1) \le |D| < (n+1)+n+(n-1)+\dots+l\]
where $D$ is the final segment of the simplicial ordering, then $|\partial_{T_n}(D)| = l$.
\end{enumerate}
\label{initialandfinal}
\end{lemma}

\begin{proof}
First, we will prove ($i$). 
Assume that 
\[1+2+3+\dots+l < |C| \le 1+2+3+\dots+(l+1).\]
Since $C$ is the initial segment of the simplicial ordering, we get that all columns from column $0$ to column $l$ must all contain a vertex that is a member of $C$. Moreover, the diagonal $|x| = n$ has no member of $C$. Thus, the vertex on top of the highest member of $C$ in each column must be in the vertex boundary of $C$. This implies that column $0$ to column $l$ has precisely $l+1$ vertices in the vertex boundary of $C$. Moreover, the vertex that is to the right of the rightmost member of $C$ in row $0$ must also be in the vertex boundary of $C$. Thus, $|\partial_{T_n}(C)| = l+2$.

Next, we will prove ($ii$).
Assume that 
\[(n+1)+n+(n-1)+\dots+(l+1) \le |D| < (n+1)+n+(n-1)+\dots+l.\] 
Since $D$ is the final segment of the simplicial ordering, the vertices in $D$ must occupy column $0$ to column $l-1$, and each column must have a vertex that is not a member of $D$. This implies that every column from column $0$ to column $l-1$ must have precisely 1 boundary vertex of $D$. Moreover, column $l$ to column $n$ must not have any boundary vertex of $D$ because all the vertices in these columns are members of $D$. Therefore, $|\partial_{T_n}(D)| = l$. \qedhere
\end{proof}

Next we will begin the proof of Theorem~\ref{thm:inequality}.

\begin{proof}[Proof of Theorem~\ref{thm:inequality}.]
Consider a triangular grid graph $T_n$. Let $j$ be the largest positive integer satisfying $1+2+3+\dots+j < \frac{|T_n|}{2}$, and let $A \subseteq V(T_n)$. It is sufficient to show that the neighborhood of $A$ is at least as large as of that of either the row packing or ice cream cone packing of size $|A|$. We will separate the proof into two cases.

\textbf{Case 1.} $|A| \le 1+2+3+\dots+j$.

Among all $B$ satisfying $|B| = |A|$ and $|N(B)| \le |N(A)|$, choose $B$ such that the sum of all the positions in simplicial of all the vertices in $B$ is minimal. 
Then $B$ must be $i$-left-compressed for all $i$, otherwise $C_i^{-}(B)$ would satisfy all of the previous conditions by Lemma~\ref{neighborhood} and would also contradict the minimality of $B$.

Next, we will show that 
$\left| N \left(B\right) \right| \ge |N(C)|$ where $C$ is the ice cream cone packing of size $|A|$. Clearly, if we view the top vertex in Figure~\ref{fig:triangulargraph} as the origin in Figure~\ref{fig:shiftedtriangulargraph}, then $C$ is exactly the initial segment of the simplicial ordering of size $|A|$. We will separate the proof into two further subcases.

\textbf{Case 1.1.} The diagonal $|x| = n$ has no member of $B$.

This case is done by Lemma~\ref{emptyfulldiagonal}.

\textbf{Case 1.2.} The diagonal $|x| = n$ has a member of $B$.

Since $|B| = |C|$, it is sufficient to show that $|\partial_{T_n}(B)| \ge |\partial_{T_n}(C)|$. First, we will determine a lower bound for $|\partial_{T_n}(B)|$.
Assume that there are $k$ members of $B$ on the diagonal $|x| = n$.
Among all vertices that are a member of $B$ and on the diagonal $|x| = n$, choose the vertex $y = (y_1,y_2)$ with the highest value of the second coordinate. Since $B$ is $i$-left-compressed for $i = 1,2$, then all vertices $z = (z_1,z_2)$ with $z_1 \le y_1$ and $z_2 \le y_2$ are members of $B$.

Since column $0$ to column $y_1-1$ all have members of $B$ and all have a vertex which is not a member of $B$, we get that each column must have at least one vertex that is in the vertex boundary of $B$. On the other hand, since there are $y_2-(k-1)$ rows between row $0$ and row $y_2-1$ that have members of $B$ and also have a vertex which is not a member of $B$, we get that each of these $y_2-(k-1)$ rows must have at least one vertex that is in the vertex boundary of $B$. Hence, the vertex boundary of $B$ has size at least $y_1+(y_2-(k-1)) = n-k+1$. 

Since 
\[|C| = \left|B\right| = |A| \le 1+2+3+\dots+j < \frac{|T_n|}{2}\]
and by Lemma~\ref{initialandfinal},
we get that $|\partial_{T_n}(C)| \le j+1$. Since we aim to show that $|\partial_{T_n}(B)| \ge |\partial_{T_n}(C)|$, it suffices to show that $j+1 \le n-k+1$.

% By definition of the vertex $z$ and $f$, it is clear that there are $f+1$ vertices in the diagonal $|x| = n$ that are members of $B{'}$. 
Let $x_1,x_2,\dots,x_k$ represent the $k$ members of $B$ on the diagonal $|x| = n$ where the second coordinates of $x_1,x_2,\dots,x_k$ are ordered in decreasing order. Define the broken line $l_i$ as the broken line that is drawn from $x_i$ straight to the left until column 
$i-1$, then straight downward to row $0$. Clearly, there are $n+2-i$ vertices on line $l_i$ and no two lines $l_i$ intersect each other. Moreover, since $B$ is $i$-left-compressed for $i = 1,2$, we get that all the vertices on every broken line $l_i$ is a member of $B$. 
Thus, 
\[\left|B\right| \ge (n+1)+n+(n-1)+\dots+(n-(k-1)).\] 
However, 
\[(n+1)+n+(n-1)+\dots+(j+1) > \frac{|T_n|}{2} > 1+2+3+\dots+j \ge |B|.\]
Hence, $j+1 < n-k+1$ as desired.

From the two subcases, we get that $\left|N\left(B\right)\right| \ge |N(C)|$.

\textbf{Case 2.} $|A| > 1+2+3+\dots+j$.
% Fix $i = 2$, and let $B = E_i(A)$.
% Clearly
% \[N(A)_t = A_t \cup A_t^{-} \cup A_t^{+} \cup A_{t-1} \cup A_{t-1}^{-} \cup A_{t+1} \cup A_{t+1}^{+}\]
% and
% \[N(B)_t = B_t \cup B_t^{-} \cup B_t^{+} \cup B_{t-1} \cup B_{t-1}^{-} \cup B_{t+1} \cup B_{t+1}^{+}.\]
% Because $B$ is an $i$-expansion of $A$, we deduce that 
% \[|B_t \cup B_t^{-} \cup B_t^{+} \cup B_{t-1} \cup B_{t-1}^{-} \cup B_{t+1} \cup B_{t+1}^{+}| \le max\{|B_t \cup B_t^{-}|, |B_{t-1} \cup B_{t-1}^{-}|, |B_{t+1}|\}.\] 
% Also $|B_t \cup B_t^{-}| \le |A_t \cup A_t^{-}|, |B_{t-1} \cup B_{t-1}^{-}| \le |A_{t-1} \cup A_{t-1}^{-}|$ and $|B_{t+1}| = |A_{t+1}|$
% Therefore, $|N(B)_t| \le |N(A)_t|$. 
% Thus, 
% \[|N(A)| = \sum_{t=1}^n |N(A)_t| \ge \sum_{t=1}^n |N(B)_t| = |N(B)|.\]

Among all $B$ satisfying $|B| = |A|$ and $|N(B)| \le |N(A)|$, choose $B$ such that the sum of all the positions in the simplicial ordering of all the vertices in $B$ is maximum. Then $B$ must be $i$-right-compressed for all $i$, otherwise $C_i^{+}(B)$ would satisfy all of the previous conditions by Lemma~\ref{neighborhood} and would also contradict the maximality of $B$.

Next, we will show that $\left|N\left(B\right)\right| \ge |N(D)|$ where $D$ is the row packing of size $|A|$. Clearly, if we view the top vertex in Figure~\ref{fig:triangulargraph} as the origin in Figure~\ref{fig:shiftedtriangulargraph}, then $D$ is exactly the final segment of the simplicial ordering of size $|A|$. We will now separate two further subcases through this knowledge.

\textbf{Case 2.1.} All vertices on the diagonal $|x| = n$ are members of $B$.

% Let $s =$ min$\{|x| : x \in B{'}\}$ and $r =$ max$\{|x| : x \notin B\}$. If $r < s$, then $B = E$ which implies that $|N(B)| \ge |N(E)|$. If $r = s$, then $\{x : |x| > r\} \subset B{'} \subset \{x : |x| \ge r$, so certainly $|N(B{'})| \ge |N(E)|$.

% Assume $r > s$. Since $B{'}$ is $i$-right-compressed for all $i$. Then $\{x : |x| = s\} \not\subset B{'}$. Thus, there exists vertices $p$ and $p{'}$ such that $|p| = |p{'}| = s$ and $p$, $p{'}$ are adjacent and $p \in B{'}$ but $p{'} \notin B{'}$. Likewise, there exists vertices $q$ and $q{'}$ such that $|q| = |q{'}| = r$ and $q$, $q{'}$ are adjacent and $q \in B{'}$ but $q{'} \notin B{'}$. 
% Set $B{''} = B{'} \cup q{'} \setminus p$.
% It can be seen that $|N(B{''})| \le |N(B{'})|$ and the sum of all the postiions in the simplicial ordering of all the vertices in $B{''}$ is greater than that of $B{''}$. This is a contradiction with the maximality of $B{'}$. Thus, the case $r > s$ can not happen.
This case is done by Lemma~\ref{emptyfulldiagonal}.

\textbf{Case 2.2.} There is a vertex on the diagonal $|x| = n$ that is not a member of $B$.

Since $|B| = |D|$, it is sufficient to show that $|\partial_{T_n}(B)| \ge |\partial_{T_n}(D)|$. First, we will determine a lower bound for $|\partial_{T_n}(B)|$.
Let $y = (y_1,y_2)$ be a vertex on the diagonal $|x| = n$ that is not a member of $B$. Since $B$ is $i$-right-compressed for $i = 1,2$, we get that all points $z = (z_1,z_2)$ with $z_1 \le y_1$ and $z_2 \le y_2$ are not members of $B$. 

We shall denote the subgraph induced by row $y_2$ to row $n$ by $G_1$, and the subgraph induced by column $y_1$ to column $n$ by $G_2$ as shown in Figure~\ref{fig: P-subgraph}. Write $B_i = V(G_i) \cap B$ for $i = 1,2$. 
Suppose that
\[1+2+3+\dots+l < |B_1| \le 1+2+3+\dots+(l+1)\] 
and 
\[1+2+3+\dots+m < |B_2| \le 1+2+3+\dots+(m+1).\]

\begin{figure}[h]
    \centering
    \begin{tikzpicture}

        \filldraw (0,0) circle(0.1cm);
	\filldraw (0,2) circle(0.1cm);
	\filldraw (2,0) circle(0.1cm);
        
        \draw (0,0) -- (0,2);
        \draw (0,0) -- (2,0);
        \draw (0,2) -- (2,0);
        \draw (1,0) -- (1,1);
        \draw (0,1) -- (1,1);

        \node at (0.5,1.2) {$G_1$};
        \node at (1.5,0.2) {$G_2$};
        \node at (-0.7,1) {row $y_2$};
        \node at (1,-0.3) {column $y_1$};
          
    \end{tikzpicture}
    \caption{A diagram of the subgraphs $G_1$ and $G_2$}
    \label{fig: P-subgraph}
\end{figure}
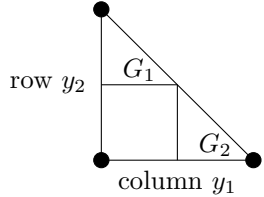

For the subgraph $G_1$ which is a triangular grid graph, view the vertex in row $n$ as the origin in the form of Figure~\ref{fig:shiftedtriangulargraph}. From this point of view, the final diagonal of $G_1$ has no member of $B_1$. Hence, Lemma~\ref{emptyfulldiagonal} can be applied to $B_1$ with respect to $G_1$. Similarly, Lemma~\ref{emptyfulldiagonal} can also be applied to $B_2$ with respect to $G_2$. By Lemmas~\ref{emptyfulldiagonal} and~\ref{initialandfinal}, we have
\[|\partial_{T_n}(B_1)| = |\partial_{G_1}(B_1)| \ge l+2\] 
and 
\[|\partial_{T_n}(B_2)| = |\partial_{G_2}(B_2)| \ge m+2.\] 
Due to the fact that $|V(G_1) \cap V(G_2)| = 1$,
we deduce that 
\[|\partial_{T_n}(B)| \ge |\partial_{T_n}(B_1)| + |\partial_{T_n}(B_2)| -1 \ge m+l+3.\]
Next, we will show that $m+l+3 \ge j+2$. Suppose for contradiction that $m+l+2 \le j$.
Since $|B| = |V(G_1) \cap B| + |V(G_2) \cap B|$, we get that 
\[|B| \le \sum_{i=1}^{l+1} i + \sum_{i=1}^{m+1} i.\]
On the other hand, 
\[\sum_{i=1}^{l+1} i + \sum_{i=1}^{m+1} i < \sum_{i=1}^{l+1} i + \sum_{i=l+2}^{m+l+2} i \le 1+2+3+\dots+j < \left|B\right|\] 
which is a contradiction. Therefore, $|\partial_{T_n}(B)| \ge m+l+3 \ge j+2$.

Next, we will show that $|\partial_{T_n}(D)| \le j+2$. Since $D$ is a final segment of the simplicial ordering and by Lemma~\ref{initialandfinal}, it is sufficient to prove that $|D| \ge (n+1)+n+(n-1)+\dots+(j+3)$. From
$1+2+3+\dots+(j+1) \ge \frac{|T_n|}{2}$ and $|T_n| = 1+2+3+\dots+(n+1)$, 
we get that 
\[(n+1)+n+(n-1)+\dots+(j+2) \le \frac{|T_n|}{2}\]
Thus, 
\[(n+1)+n+(n-1)+\dots+(j+3) \le \frac{|T_n|}{2}-(j+2) < \frac{|T_n|}{2}-(j+1) \le 1+2+3+\dots+j < |B{'}| = |E|.\] 
Hence, $|\partial_{T_n}(E)| \le j+2$.

From the two subcases, we get that $\left|N\left(B\right)\right| \ge |N(E)|$.

Combining the results of Case 1 and Case 2, we prove that the beighborhood of $A$ is at least as large as the minimum of that of the row packing and ice cream cone packing of size $|A|$.
\end{proof}

%%%%%%%%%%%%%%%%%%%%
\section{Proof of Theorem~\ref{thm:search}}
\label{sec:thm:search}

In this section we will provide lower and upper bounds for the inspection number of the triangular grid graph.

\begin{proof}[Proof of  Theorem~\ref{thm:search}]

First, we will show that 
$\In(T_n) \le \left\lceil \frac{3n}{4} \right\rceil + 2$ by demonstrating a sequence of searches on at most 
$k = \left\lceil \frac{3n}{4} \right\rceil + 2$ vertices
that clears the intruder from $T_n$. After a search, we say that a vertex is $\emph{fully cleared}$ if the intruder cannot move to the vertex immediately after the search.

We will separate the sequence of searches into three stages. In the first stage, we will ensure that every vertex in row $n$ to row $n-k+3$ is fully cleared. First, we examine the vertex in row $n$ and the two vertices in row $n-1$. After this search, the vertex in row $n$ becomes fully cleared. 
Once we have made every vertex in row $n$ to row $n-i$ fully cleared, we will make every vertex in row $n-i-1$ fully cleared. Let the vertices in row $n-i-1$ be denoted by $v_1,v_2,\dots,v_{i+2}$ from left to right.
Similarly, denote the vertices in row $n-i-2$ by $v_{i+3},v_{i+4},\dots,v_{2i+5}$ from left to right.
We start the phase of fully clearing the vertices in row $n-i-1$ by examining $v_1,v_2,\dots,v_{i+4}$. This search will make $v_1$ become fully cleared. 
In the $j$-th turn of this phase for 
$j \in \{1,2,\dots,i+2\}$, we examine $v_j,v_{j+1},\dots,v_{j+i+3}$. 
Clearly, the vertices $v_1,v_2,\dots,v_j$ will be fully cleared after the $j$-th search.
Thus, all the vertices in row $n-i-1$ will become fully cleared after $i+2$ turns. Evidently, the final row that can become fully cleared using this strategy is row $n-k+3$ since row 
$n-k+3$ has precisely $k-2$ vertices. Figure~\ref{fig:firststep} demonstrates the first stage of fully clearing $T_5$.

\begin{figure}[h]
	\centering
        \scalebox{0.55}{
	\begin{tikzpicture}
		\draw[fill = yellow] (0,5) circle(0.1cm);
		\draw[fill = yellow] (0,4) circle(0.1cm);
		\draw[fill = yellow] (1,4) circle(0.1cm);
		\draw[fill = red] (0,3) circle(0.1cm);
		\draw[fill = red] (1,3) circle(0.1cm);
            \draw[fill = red] (2,3) circle(0.1cm);
            \draw[fill = red] (0,2) circle(0.1cm);
            \draw[fill = red] (1,2) circle(0.1cm);
            \draw[fill = red] (2,2) circle(0.1cm);
            \draw[fill = red] (3,2) circle(0.1cm);
            \draw[fill = red] (0,1) circle(0.1cm);
		\draw[fill = red] (1,1) circle(0.1cm);
		\draw[fill = red] (2,1) circle(0.1cm);
		\draw[fill = red] (3,1) circle(0.1cm);
		\draw[fill = red] (4,1) circle(0.1cm);
            \draw[fill = red] (0,0) circle(0.1cm);
            \draw[fill = red] (1,0) circle(0.1cm);
            \draw[fill = red] (2,0) circle(0.1cm);
            \draw[fill = red] (3,0) circle(0.1cm);
            \draw[fill = red] (4,0) circle(0.1cm);
            \draw[fill = red] (5,0) circle(0.1cm);
           
            \draw (0,0) -- (5,0);
		\draw (0,0) -- (0,5);
		\draw (0,1) -- (1,0);
		\draw (0,2) -- (2,0);
            \draw (0,3) -- (3,0);
            \draw (0,4) -- (4,0);
            \draw (0,5) -- (5,0);
            \draw (0,4) -- (1,4);
            \draw (0,3) -- (2,3);
            \draw (0,2) -- (3,2);
		\draw (0,1) -- (4,1);
		\draw (1,0) -- (1,4);
		\draw (2,0) -- (2,3);
            \draw (3,0) -- (3,2);
            \draw (4,0) -- (4,1);

            \draw[fill = green] (6,5) circle(0.1cm);
		\draw[fill = yellow] (6,4) circle(0.1cm);
		\draw[fill = yellow] (7,4) circle(0.1cm);
		\draw[fill = yellow] (6,3) circle(0.1cm);
		\draw[fill = yellow] (7,3) circle(0.1cm);
            \draw[fill = red] (8,3) circle(0.1cm);
            \draw[fill = red] (6,2) circle(0.1cm);
            \draw[fill = red] (7,2) circle(0.1cm);
            \draw[fill = red] (8,2) circle(0.1cm);
            \draw[fill = red] (9,2) circle(0.1cm);
            \draw[fill = red] (6,1) circle(0.1cm);
		\draw[fill = red] (7,1) circle(0.1cm);
		\draw[fill = red] (8,1) circle(0.1cm);
		\draw[fill = red] (9,1) circle(0.1cm);
		\draw[fill = red] (10,1) circle(0.1cm);
            \draw[fill = red] (6,0) circle(0.1cm);
            \draw[fill = red] (7,0) circle(0.1cm);
            \draw[fill = red] (8,0) circle(0.1cm);
            \draw[fill = red] (9,0) circle(0.1cm);
            \draw[fill = red] (10,0) circle(0.1cm);
            \draw[fill = red] (11,0) circle(0.1cm);
  
		\draw (6,0) -- (11,0);
		\draw (6,0) -- (6,5);
		\draw (6,1) -- (7,0);
		\draw (6,2) -- (8,0);
            \draw (6,3) -- (9,0);
            \draw (6,4) -- (10,0);
            \draw (6,5) -- (11,0);
            \draw (6,4) -- (7,4);
            \draw (6,3) -- (8,3);
            \draw (6,2) -- (9,2);
		\draw (6,1) -- (10,1);
		\draw (7,0) -- (7,4);
		\draw (8,0) -- (8,3);
            \draw (9,0) -- (9,2);
            \draw (10,0) -- (10,1);

            \draw[fill = green] (12,5) circle(0.1cm);
		\draw[fill = green] (12,4) circle(0.1cm);
		\draw[fill = yellow] (13,4) circle(0.1cm);
		\draw[fill = yellow] (12,3) circle(0.1cm);
		\draw[fill = yellow] (13,3) circle(0.1cm);
            \draw[fill = yellow] (14,3) circle(0.1cm);
            \draw[fill = red] (12,2) circle(0.1cm);
            \draw[fill = red] (13,2) circle(0.1cm);
            \draw[fill = red] (14,2) circle(0.1cm);
            \draw[fill = red] (15,2) circle(0.1cm);
            \draw[fill = red] (12,1) circle(0.1cm);
		\draw[fill = red] (13,1) circle(0.1cm);
		\draw[fill = red] (14,1) circle(0.1cm);
		\draw[fill = red] (15,1) circle(0.1cm);
		\draw[fill = red] (16,1) circle(0.1cm);
            \draw[fill = red] (12,0) circle(0.1cm);
            \draw[fill = red] (13,0) circle(0.1cm);
            \draw[fill = red] (14,0) circle(0.1cm);
            \draw[fill = red] (15,0) circle(0.1cm);
            \draw[fill = red] (16,0) circle(0.1cm);
            \draw[fill = red] (17,0) circle(0.1cm);
  
		\draw (12,0) -- (17,0);
		\draw (12,0) -- (12,5);
		\draw (12,1) -- (13,0);
		\draw (12,2) -- (14,0);
            \draw (12,3) -- (15,0);
            \draw (12,4) -- (16,0);
            \draw (12,5) -- (17,0);
            \draw (12,4) -- (13,4);
            \draw (12,3) -- (14,3);
            \draw (12,2) -- (15,2);
		\draw (12,1) -- (16,1);
		\draw (13,0) -- (13,4);
		\draw (14,0) -- (14,3);
            \draw (15,0) -- (15,2);
            \draw (16,0) -- (16,1);

            \draw[fill = green] (0,-1) circle(0.1cm);
		\draw[fill = green] (0,-2) circle(0.1cm);
		\draw[fill = green] (1,-2) circle(0.1cm);
		\draw[fill = yellow] (0,-3) circle(0.1cm);
		\draw[fill = yellow] (1,-3) circle(0.1cm);
            \draw[fill = yellow] (2,-3) circle(0.1cm);
            \draw[fill = yellow] (0,-4) circle(0.1cm);
            \draw[fill = yellow] (1,-4) circle(0.1cm);
            \draw[fill = red] (2,-4) circle(0.1cm);
            \draw[fill = red] (3,-4) circle(0.1cm);
            \draw[fill = red] (0,-5) circle(0.1cm);
		\draw[fill = red] (1,-5) circle(0.1cm);
		\draw[fill = red] (2,-5) circle(0.1cm);
		\draw[fill = red] (3,-5) circle(0.1cm);
		\draw[fill = red] (4,-5) circle(0.1cm);
            \draw[fill = red] (0,-6) circle(0.1cm);
            \draw[fill = red] (1,-6) circle(0.1cm);
            \draw[fill = red] (2,-6) circle(0.1cm);
            \draw[fill = red] (3,-6) circle(0.1cm);
            \draw[fill = red] (4,-6) circle(0.1cm);
            \draw[fill = red] (5,-6) circle(0.1cm);
           
            \draw (0,-6) -- (5,-6);
		\draw (0,-6) -- (0,-1);
		\draw (0,-5) -- (1,-6);
		\draw (0,-4) -- (2,-6);
            \draw (0,-3) -- (3,-6);
            \draw (0,-2) -- (4,-6);
            \draw (0,-1) -- (5,-6);
            \draw (0,-2) -- (1,-2);
            \draw (0,-3) -- (2,-3);
            \draw (0,-4) -- (3,-4);
		\draw (0,-5) -- (4,-5);
		\draw (1,-6) -- (1,-2);
		\draw (2,-6) -- (2,-3);
            \draw (3,-6) -- (3,-4);
            \draw (4,-6) -- (4,-5);

            \draw[fill = green] (6,-1) circle(0.1cm);
		\draw[fill = green] (6,-2) circle(0.1cm);
		\draw[fill = green] (7,-2) circle(0.1cm);
		\draw[fill = green] (6,-3) circle(0.1cm);
		\draw[fill = yellow] (7,-3) circle(0.1cm);
            \draw[fill = yellow] (8,-3) circle(0.1cm);
            \draw[fill = yellow] (6,-4) circle(0.1cm);
            \draw[fill = yellow] (7,-4) circle(0.1cm);
            \draw[fill = yellow] (8,-4) circle(0.1cm);
            \draw[fill = red] (9,-4) circle(0.1cm);
            \draw[fill = red] (6,-5) circle(0.1cm);
		\draw[fill = red] (7,-5) circle(0.1cm);
		\draw[fill = red] (8,-5) circle(0.1cm);
		\draw[fill = red] (9,-5) circle(0.1cm);
		\draw[fill = red] (10,-5) circle(0.1cm);
            \draw[fill = red] (6,-6) circle(0.1cm);
            \draw[fill = red] (7,-6) circle(0.1cm);
            \draw[fill = red] (8,-6) circle(0.1cm);
            \draw[fill = red] (9,-6) circle(0.1cm);
            \draw[fill = red] (10,-6) circle(0.1cm);
            \draw[fill = red] (11,-6) circle(0.1cm);
           
            \draw (6,-6) -- (11,-6);
		\draw (6,-6) -- (6,-1);
		\draw (6,-5) -- (7,-6);
		\draw (6,-4) -- (8,-6);
            \draw (6,-3) -- (9,-6);
            \draw (6,-2) -- (10,-6);
            \draw (6,-1) -- (11,-6);
            \draw (6,-2) -- (7,-2);
            \draw (6,-3) -- (8,-3);
            \draw (6,-4) -- (9,-4);
		\draw (6,-5) -- (10,-5);
		\draw (7,-6) -- (7,-2);
		\draw (8,-6) -- (8,-3);
            \draw (9,-6) -- (9,-4);
            \draw (10,-6) -- (10,-5);

            \draw[fill = green] (12,-1) circle(0.1cm);
		\draw[fill = green] (12,-2) circle(0.1cm);
		\draw[fill = green] (13,-2) circle(0.1cm);
		\draw[fill = green] (12,-3) circle(0.1cm);
		\draw[fill = green] (13,-3) circle(0.1cm);
            \draw[fill = yellow] (14,-3) circle(0.1cm);
            \draw[fill = yellow] (12,-4) circle(0.1cm);
            \draw[fill = yellow] (13,-4) circle(0.1cm);
            \draw[fill = yellow] (14,-4) circle(0.1cm);
            \draw[fill = yellow] (15,-4) circle(0.1cm);
            \draw[fill = red] (12,-5) circle(0.1cm);
		\draw[fill = red] (13,-5) circle(0.1cm);
		\draw[fill = red] (14,-5) circle(0.1cm);
		\draw[fill = red] (15,-5) circle(0.1cm);
		\draw[fill = red] (16,-5) circle(0.1cm);
            \draw[fill = red] (12,-6) circle(0.1cm);
            \draw[fill = red] (13,-6) circle(0.1cm);
            \draw[fill = red] (14,-6) circle(0.1cm);
            \draw[fill = red] (15,-6) circle(0.1cm);
            \draw[fill = red] (16,-6) circle(0.1cm);
            \draw[fill = red] (17,-6) circle(0.1cm);
           
            \draw (12,-6) -- (17,-6);
		\draw (12,-6) -- (12,-1);
		\draw (12,-5) -- (13,-6);
		\draw (12,-4) -- (14,-6);
            \draw (12,-3) -- (15,-6);
            \draw (12,-2) -- (16,-6);
            \draw (12,-1) -- (17,-6);
            \draw (12,-2) -- (13,-2);
            \draw (12,-3) -- (14,-3);
            \draw (12,-4) -- (15,-4);
		\draw (12,-5) -- (16,-5);
		\draw (13,-6) -- (13,-2);
		\draw (14,-6) -- (14,-3);
            \draw (15,-6) -- (15,-4);
            \draw (16,-6) -- (16,-5);

            \draw[fill = green] (0,-7) circle(0.1cm);
		\draw[fill = green] (0,-8) circle(0.1cm);
		\draw[fill = green] (1,-8) circle(0.1cm);
		\draw[fill = green] (0,-9) circle(0.1cm);
		\draw[fill = green] (1,-9) circle(0.1cm);
            \draw[fill = green] (2,-9) circle(0.1cm);
            \draw[fill = yellow] (0,-10) circle(0.1cm);
            \draw[fill = yellow] (1,-10) circle(0.1cm);
            \draw[fill = yellow] (2,-10) circle(0.1cm);
            \draw[fill = yellow] (3,-10) circle(0.1cm);
            \draw[fill = yellow] (0,-11) circle(0.1cm);
		\draw[fill = yellow] (1,-11) circle(0.1cm);
		\draw[fill = red] (2,-11) circle(0.1cm);
		\draw[fill = red] (3,-11) circle(0.1cm);
		\draw[fill = red] (4,-11) circle(0.1cm);
            \draw[fill = red] (0,-12) circle(0.1cm);
            \draw[fill = red] (1,-12) circle(0.1cm);
            \draw[fill = red] (2,-12) circle(0.1cm);
            \draw[fill = red] (3,-12) circle(0.1cm);
            \draw[fill = red] (4,-12) circle(0.1cm);
            \draw[fill = red] (5,-12) circle(0.1cm);
           
            \draw (0,-12) -- (5,-12);
		\draw (0,-12) -- (0,-7);
		\draw (0,-11) -- (1,-11);
		\draw (0,-10) -- (2,-12);
            \draw (0,-9) -- (3,-12);
            \draw (0,-8) -- (4,-12);
            \draw (0,-7) -- (5,-12);
            \draw (0,-8) -- (1,-8);
            \draw (0,-9) -- (2,-9);
            \draw (0,-10) -- (3,-10);
		\draw (0,-11) -- (4,-11);
		\draw (1,-12) -- (1,-8);
		\draw (2,-12) -- (2,-9);
            \draw (3,-12) -- (3,-10);
            \draw (4,-12) -- (4,-11);

            \draw[fill = green] (6,-7) circle(0.1cm);
		\draw[fill = green] (6,-8) circle(0.1cm);
		\draw[fill = green] (7,-8) circle(0.1cm);
		\draw[fill = green] (6,-9) circle(0.1cm);
		\draw[fill = green] (7,-9) circle(0.1cm);
            \draw[fill = green] (8,-9) circle(0.1cm);
            \draw[fill = green] (6,-10) circle(0.1cm);
            \draw[fill = yellow] (7,-10) circle(0.1cm);
            \draw[fill = yellow] (8,-10) circle(0.1cm);
            \draw[fill = yellow] (9,-10) circle(0.1cm);
            \draw[fill = yellow] (6,-11) circle(0.1cm);
		\draw[fill = yellow] (7,-11) circle(0.1cm);
		\draw[fill = yellow] (8,-11) circle(0.1cm);
		\draw[fill = red] (9,-11) circle(0.1cm);
		\draw[fill = red] (10,-11) circle(0.1cm);
            \draw[fill = red] (6,-12) circle(0.1cm);
            \draw[fill = red] (7,-12) circle(0.1cm);
            \draw[fill = red] (8,-12) circle(0.1cm);
            \draw[fill = red] (9,-12) circle(0.1cm);
            \draw[fill = red] (10,-12) circle(0.1cm);
            \draw[fill = red] (11,-12) circle(0.1cm);
           
            \draw (6,-12) -- (11,-12);
		\draw (6,-12) -- (6,-7);
		\draw (6,-11) -- (7,-11);
		\draw (6,-10) -- (8,-12);
            \draw (6,-9) -- (9,-12);
            \draw (6,-8) -- (10,-12);
            \draw (6,-7) -- (11,-12);
            \draw (6,-8) -- (7,-8);
            \draw (6,-9) -- (8,-9);
            \draw (6,-10) -- (9,-10);
		\draw (6,-11) -- (10,-11);
		\draw (7,-12) -- (7,-8);
		\draw (8,-12) -- (8,-9);
            \draw (9,-12) -- (9,-10);
            \draw (10,-12) -- (10,-11);

            \draw[fill = green] (12,-7) circle(0.1cm);
		\draw[fill = green] (12,-8) circle(0.1cm);
		\draw[fill = green] (13,-8) circle(0.1cm);
		\draw[fill = green] (12,-9) circle(0.1cm);
		\draw[fill = green] (13,-9) circle(0.1cm);
            \draw[fill = green] (14,-9) circle(0.1cm);
            \draw[fill = green] (12,-10) circle(0.1cm);
            \draw[fill = green] (13,-10) circle(0.1cm);
            \draw[fill = yellow] (14,-10) circle(0.1cm);
            \draw[fill = yellow] (15,-10) circle(0.1cm);
            \draw[fill = yellow] (12,-11) circle(0.1cm);
		\draw[fill = yellow] (13,-11) circle(0.1cm);
		\draw[fill = yellow] (14,-11) circle(0.1cm);
		\draw[fill = yellow] (15,-11) circle(0.1cm);
		\draw[fill = red] (16,-11) circle(0.1cm);
            \draw[fill = red] (12,-12) circle(0.1cm);
            \draw[fill = red] (13,-12) circle(0.1cm);
            \draw[fill = red] (14,-12) circle(0.1cm);
            \draw[fill = red] (15,-12) circle(0.1cm);
            \draw[fill = red] (16,-12) circle(0.1cm);
            \draw[fill = red] (17,-12) circle(0.1cm);
           
            \draw (12,-12) -- (17,-12);
		\draw (12,-12) -- (12,-7);
		\draw (12,-11) -- (13,-11);
		\draw (12,-10) -- (14,-12);
            \draw (12,-9) -- (15,-12);
            \draw (12,-8) -- (16,-12);
            \draw (12,-7) -- (17,-12);
            \draw (12,-8) -- (13,-8);
            \draw (12,-9) -- (14,-9);
            \draw (12,-10) -- (15,-10);
		\draw (12,-11) -- (16,-11);
		\draw (13,-12) -- (13,-8);
		\draw (14,-12) -- (14,-9);
            \draw (15,-12) -- (15,-10);
            \draw (16,-12) -- (16,-11);

            \draw[fill = green] (0,-13) circle(0.1cm);
		\draw[fill = green] (0,-14) circle(0.1cm);
		\draw[fill = green] (1,-14) circle(0.1cm);
		\draw[fill = green] (0,-15) circle(0.1cm);
		\draw[fill = green] (1,-15) circle(0.1cm);
            \draw[fill = green] (2,-15) circle(0.1cm);
            \draw[fill = green] (0,-16) circle(0.1cm);
            \draw[fill = green] (1,-16) circle(0.1cm);
            \draw[fill = green] (2,-16) circle(0.1cm);
            \draw[fill = yellow] (3,-16) circle(0.1cm);
            \draw[fill = yellow] (0,-17) circle(0.1cm);
		\draw[fill = yellow] (1,-17) circle(0.1cm);
		\draw[fill = yellow] (2,-17) circle(0.1cm);
		\draw[fill = yellow] (3,-17) circle(0.1cm);
		\draw[fill = yellow] (4,-17) circle(0.1cm);
            \draw[fill = red] (0,-18) circle(0.1cm);
            \draw[fill = red] (1,-18) circle(0.1cm);
            \draw[fill = red] (2,-18) circle(0.1cm);
            \draw[fill = red] (3,-18) circle(0.1cm);
            \draw[fill = red] (4,-18) circle(0.1cm);
            \draw[fill = red] (5,-18) circle(0.1cm);
           
            \draw (0,-18) -- (5,-18);
		\draw (0,-18) -- (0,-13);
		\draw (0,-17) -- (1,-17);
		\draw (0,-16) -- (2,-18);
            \draw (0,-15) -- (3,-18);
            \draw (0,-14) -- (4,-18);
            \draw (0,-13) -- (5,-18);
            \draw (0,-14) -- (1,-14);
            \draw (0,-15) -- (2,-15);
            \draw (0,-16) -- (3,-16);
		\draw (0,-17) -- (4,-17);
		\draw (1,-18) -- (1,-14);
		\draw (2,-18) -- (2,-15);
            \draw (3,-18) -- (3,-16);
            \draw (4,-18) -- (4,-17);

        \end{tikzpicture}
        }

        \caption{The first stage of fully clearing $T_5$. Green vertices represent vertices that are fully cleared. Yellow vertices represent vertices that are examined during each turn. Red vertices represent vertices that are neither fully cleared, nor being searched.}
	\label{fig:firststep}
\end{figure}
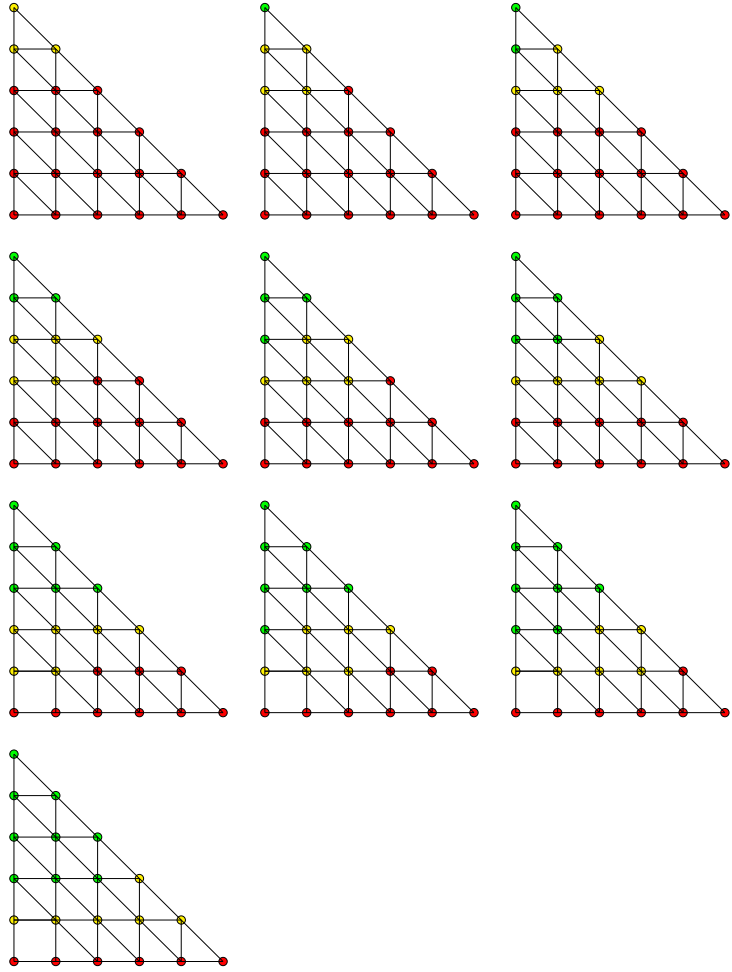

Next, we will demonstrate the second stage of our search strategy. Our goal is to fully clear every vertex between column $0$ and column $n-k+1$, but we allow the previously fully cleared vertices to the right of column 
$n-k+1$ to lose their state.
Let $R_1, R_2,\dots, R_{n-k+3}$ be the first $n-k+2$ vertices from the left of row $n-k+2$, row $n-k+1$,\dots, row $0$ respectively. Let $v_1, v_2,\dots, v_{n-k+2}$ be the $(2n-2k+5)$-th, $(2n-2k+4)$-th,\dots, $(n-k+4)$-th vertex of row $n-k+2$ respectively.
We also denote the bottom $n-k+3$ vertices of column $n-k+2$ by 
$v_{n-k+3},v_{n-k+4},\dots,
v_{2n-2k+5}$ from top to bottom.
On the first move of this stage, we examine $3n-3k+8 \le k$ vertices consisting of $v_1,v_2,\dots,v_{n-k+4}$
and every vertex in $R_1$ and $R_2$. Clearly, this move will make $R_1$ fully cleared, while the vertices above row $n-k+2$ and to the left of column $2n-2k+4$ remain fully cleared.
In the $j$-th turn of the second stage for $j = 1,2,\dots,n-k+2$, we examine $v_j,v_{j+1},\dots,
v_{j+n-k+3}$ and every vertex in $R_j$ and $R_{j+1}$. 
Clearly, every vertex in $R_1,R_2,\dots,R_j$ will be fully cleared after this search, while the vertices above row $n-k+2$ and to the left of column $2n-2k-j+5$ remain fully cleared.
Moreover, in the last turn
all the vertices in 
$R_{n-k+3}$ will also become fully cleared. 
Hence, all the vertices in column $0$ to column $n-k+1$ become fully cleared after turn $n-k+2$ of the second stage. 
Figure~\ref{fig:secondstep} demonstrates the second stage of fully clearing $T_5$.

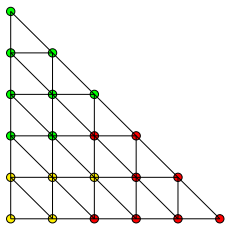
\begin{figure}[h]
	\centering
        \scalebox{0.55}{
	\begin{tikzpicture}{scale=0.1}
		\draw[fill = green] (0,5) circle(0.1cm);
		\draw[fill = green] (0,4) circle(0.1cm);
		\draw[fill = green] (1,4) circle(0.1cm);
		\draw[fill = green] (0,3) circle(0.1cm);
		\draw[fill = green] (1,3) circle(0.1cm);
            \draw[fill = green] (2,3) circle(0.1cm);
            \draw[fill = green] (0,2) circle(0.1cm);
            \draw[fill = green] (1,2) circle(0.1cm);
            \draw[fill = red] (2,2) circle(0.1cm);
            \draw[fill = red] (3,2) circle(0.1cm);
            \draw[fill = yellow] (0,1) circle(0.1cm);
		\draw[fill = yellow] (1,1) circle(0.1cm);
		\draw[fill = yellow] (2,1) circle(0.1cm);
		\draw[fill = red] (3,1) circle(0.1cm);
		\draw[fill = red] (4,1) circle(0.1cm);
            \draw[fill = yellow] (0,0) circle(0.1cm);
            \draw[fill = yellow] (1,0) circle(0.1cm);
            \draw[fill = red] (2,0) circle(0.1cm);
            \draw[fill = red] (3,0) circle(0.1cm);
            \draw[fill = red] (4,0) circle(0.1cm);
            \draw[fill = red] (5,0) circle(0.1cm);
           
            \draw (0,0) -- (5,0);
		\draw (0,0) -- (0,5);
		\draw (0,1) -- (1,0);
		\draw (0,2) -- (2,0);
            \draw (0,3) -- (3,0);
            \draw (0,4) -- (4,0);
            \draw (0,5) -- (5,0);
            \draw (0,4) -- (1,4);
            \draw (0,3) -- (2,3);
            \draw (0,2) -- (3,2);
		\draw (0,1) -- (4,1);
		\draw (1,0) -- (1,4);
		\draw (2,0) -- (2,3);
            \draw (3,0) -- (3,2);
            \draw (4,0) -- (4,1);

        \end{tikzpicture}
        }

        \caption{The second stage of fully clearing $T_5$.}
	\label{fig:secondstep}
\end{figure}

The third and final step of our search method is to fully clear the remaining uncleared columns from the second stage, i.e. columns $n-k+2$ to $n$. This third stage is essentially a reverse of the first stage. First, we will make every vertex in column 
$n-k+2$ 
fully cleared. Denote the vertices in column 
$n-k+2$ 
by 
$v_1,v_2,\dots,v_{k-1}$ 
from top to bottom.
Similarly, denote the vertices in column 
$n-k+3$ by $v_{k},v_{k+1},\dots,
v_{2k-3}$ from top to bottom.
We start the phase of fully clearing column 
$n-k+2$ by examining 
$v_1,v_2,\dots,v_{k}$. This search will make $v_1$ become fully cleared. 
In the $j$-th turn of this phase for 
$j \in \{1,2,\dots,k-2\}$, we examine $v_j,v_{j+1},\dots,v_{j+k-1}$.
Clearly, vertices $v_1,v_2,\dots,v_j$ will become fully cleared after the $j$-th search. Moreover, $v_{k-1}$ will also be fully cleared after turn $k-2$ since all the vertices in column $n-k+3$ will be examined during this turn.
Thus, all the vertices in column 
$n-k+2$ will become fully cleared after $k-2$ turns.
Once we have made every vertex in column $n-k+2$ to column $n-k+2+i$ fully cleared, we will make every vertex in column 
$n-k+3+i$ fully cleared by using the same strategy described for column $n-k+2$, but examining 
$k-1-i$ vertices instead of $k$ vertices. Evidently, this strategy can be utilized until the entire triangular grid graph has become fully cleared. Thus, $\In(T_n) \le k = \lceil\frac{3n}{4}\rceil+2$. 
Figure~\ref{fig:thirdstep} demonstrates the final stage of fully clearing $T_5$.

\begin{figure}[h]
	\centering
        \scalebox{0.55}{
	\begin{tikzpicture}
		\draw[fill = green] (0,5) circle(0.1cm);
		\draw[fill = green] (0,4) circle(0.1cm);
		\draw[fill = yellow] (1,4) circle(0.1cm);
		\draw[fill = green] (0,3) circle(0.1cm);
		\draw[fill = yellow] (1,3) circle(0.1cm);
            \draw[fill = yellow] (2,3) circle(0.1cm);
            \draw[fill = green] (0,2) circle(0.1cm);
            \draw[fill = yellow] (1,2) circle(0.1cm);
            \draw[fill = red] (2,2) circle(0.1cm);
            \draw[fill = red] (3,2) circle(0.1cm);
            \draw[fill = green] (0,1) circle(0.1cm);
		\draw[fill = yellow] (1,1) circle(0.1cm);
		\draw[fill = red] (2,1) circle(0.1cm);
		\draw[fill = red] (3,1) circle(0.1cm);
		\draw[fill = red] (4,1) circle(0.1cm);
            \draw[fill = green] (0,0) circle(0.1cm);
            \draw[fill = yellow] (1,0) circle(0.1cm);
            \draw[fill = red] (2,0) circle(0.1cm);
            \draw[fill = red] (3,0) circle(0.1cm);
            \draw[fill = red] (4,0) circle(0.1cm);
            \draw[fill = red] (5,0) circle(0.1cm);
           
            \draw (0,0) -- (5,0);
		\draw (0,0) -- (0,5);
		\draw (0,1) -- (1,0);
		\draw (0,2) -- (2,0);
            \draw (0,3) -- (3,0);
            \draw (0,4) -- (4,0);
            \draw (0,5) -- (5,0);
            \draw (0,4) -- (1,4);
            \draw (0,3) -- (2,3);
            \draw (0,2) -- (3,2);
		\draw (0,1) -- (4,1);
		\draw (1,0) -- (1,4);
		\draw (2,0) -- (2,3);
            \draw (3,0) -- (3,2);
            \draw (4,0) -- (4,1);

            \draw[fill = green] (6,5) circle(0.1cm);
		\draw[fill = green] (6,4) circle(0.1cm);
		\draw[fill = green] (7,4) circle(0.1cm);
		\draw[fill = green] (6,3) circle(0.1cm);
		\draw[fill = yellow] (7,3) circle(0.1cm);
            \draw[fill = yellow] (8,3) circle(0.1cm);
            \draw[fill = green] (6,2) circle(0.1cm);
            \draw[fill = yellow] (7,2) circle(0.1cm);
            \draw[fill = yellow] (8,2) circle(0.1cm);
            \draw[fill = red] (9,2) circle(0.1cm);
            \draw[fill = green] (6,1) circle(0.1cm);
		\draw[fill = yellow] (7,1) circle(0.1cm);
		\draw[fill = red] (8,1) circle(0.1cm);
		\draw[fill = red] (9,1) circle(0.1cm);
		\draw[fill = red] (10,1) circle(0.1cm);
            \draw[fill = green] (6,0) circle(0.1cm);
            \draw[fill = yellow] (7,0) circle(0.1cm);
            \draw[fill = red] (8,0) circle(0.1cm);
            \draw[fill = red] (9,0) circle(0.1cm);
            \draw[fill = red] (10,0) circle(0.1cm);
            \draw[fill = red] (11,0) circle(0.1cm);
  
		\draw (6,0) -- (11,0);
		\draw (6,0) -- (6,5);
		\draw (6,1) -- (7,0);
		\draw (6,2) -- (8,0);
            \draw (6,3) -- (9,0);
            \draw (6,4) -- (10,0);
            \draw (6,5) -- (11,0);
            \draw (6,4) -- (7,4);
            \draw (6,3) -- (8,3);
            \draw (6,2) -- (9,2);
		\draw (6,1) -- (10,1);
		\draw (7,0) -- (7,4);
		\draw (8,0) -- (8,3);
            \draw (9,0) -- (9,2);
            \draw (10,0) -- (10,1);

            \draw[fill = green] (12,5) circle(0.1cm);
		\draw[fill = green] (12,4) circle(0.1cm);
		\draw[fill = green] (13,4) circle(0.1cm);
		\draw[fill = green] (12,3) circle(0.1cm);
		\draw[fill = green] (13,3) circle(0.1cm);
            \draw[fill = yellow] (14,3) circle(0.1cm);
            \draw[fill = green] (12,2) circle(0.1cm);
            \draw[fill = yellow] (13,2) circle(0.1cm);
            \draw[fill = yellow] (14,2) circle(0.1cm);
            \draw[fill = red] (15,2) circle(0.1cm);
            \draw[fill = green] (12,1) circle(0.1cm);
		\draw[fill = yellow] (13,1) circle(0.1cm);
		\draw[fill = yellow] (14,1) circle(0.1cm);
		\draw[fill = red] (15,1) circle(0.1cm);
		\draw[fill = red] (16,1) circle(0.1cm);
            \draw[fill = green] (12,0) circle(0.1cm);
            \draw[fill = yellow] (13,0) circle(0.1cm);
            \draw[fill = red] (14,0) circle(0.1cm);
            \draw[fill = red] (15,0) circle(0.1cm);
            \draw[fill = red] (16,0) circle(0.1cm);
            \draw[fill = red] (17,0) circle(0.1cm);
  
		\draw (12,0) -- (17,0);
		\draw (12,0) -- (12,5);
		\draw (12,1) -- (13,0);
		\draw (12,2) -- (14,0);
            \draw (12,3) -- (15,0);
            \draw (12,4) -- (16,0);
            \draw (12,5) -- (17,0);
            \draw (12,4) -- (13,4);
            \draw (12,3) -- (14,3);
            \draw (12,2) -- (15,2);
		\draw (12,1) -- (16,1);
		\draw (13,0) -- (13,4);
		\draw (14,0) -- (14,3);
            \draw (15,0) -- (15,2);
            \draw (16,0) -- (16,1);

            \draw[fill = green] (0,-1) circle(0.1cm);
		\draw[fill = green] (0,-2) circle(0.1cm);
		\draw[fill = green] (1,-2) circle(0.1cm);
		\draw[fill = green] (0,-3) circle(0.1cm);
		\draw[fill = green] (1,-3) circle(0.1cm);
            \draw[fill = yellow] (2,-3) circle(0.1cm);
            \draw[fill = green] (0,-4) circle(0.1cm);
            \draw[fill = green] (1,-4) circle(0.1cm);
            \draw[fill = yellow] (2,-4) circle(0.1cm);
            \draw[fill = red] (3,-4) circle(0.1cm);
            \draw[fill = green] (0,-5) circle(0.1cm);
		\draw[fill = yellow] (1,-5) circle(0.1cm);
		\draw[fill = yellow] (2,-5) circle(0.1cm);
		\draw[fill = red] (3,-5) circle(0.1cm);
		\draw[fill = red] (4,-5) circle(0.1cm);
            \draw[fill = green] (0,-6) circle(0.1cm);
            \draw[fill = yellow] (1,-6) circle(0.1cm);
            \draw[fill = yellow] (2,-6) circle(0.1cm);
            \draw[fill = red] (3,-6) circle(0.1cm);
            \draw[fill = red] (4,-6) circle(0.1cm);
            \draw[fill = red] (5,-6) circle(0.1cm);
           
            \draw (0,-6) -- (5,-6);
		\draw (0,-6) -- (0,-1);
		\draw (0,-5) -- (1,-6);
		\draw (0,-4) -- (2,-6);
            \draw (0,-3) -- (3,-6);
            \draw (0,-2) -- (4,-6);
            \draw (0,-1) -- (5,-6);
            \draw (0,-2) -- (1,-2);
            \draw (0,-3) -- (2,-3);
            \draw (0,-4) -- (3,-4);
		\draw (0,-5) -- (4,-5);
		\draw (1,-6) -- (1,-2);
		\draw (2,-6) -- (2,-3);
            \draw (3,-6) -- (3,-4);
            \draw (4,-6) -- (4,-5);

            \draw[fill = green] (6,-1) circle(0.1cm);
		\draw[fill = green] (6,-2) circle(0.1cm);
		\draw[fill = green] (7,-2) circle(0.1cm);
		\draw[fill = green] (6,-3) circle(0.1cm);
		\draw[fill = green] (7,-3) circle(0.1cm);
            \draw[fill = yellow] (8,-3) circle(0.1cm);
            \draw[fill = green] (6,-4) circle(0.1cm);
            \draw[fill = green] (7,-4) circle(0.1cm);
            \draw[fill = yellow] (8,-4) circle(0.1cm);
            \draw[fill = yellow] (9,-4) circle(0.1cm);
            \draw[fill = green] (6,-5) circle(0.1cm);
		\draw[fill = green] (7,-5) circle(0.1cm);
		\draw[fill = yellow] (8,-5) circle(0.1cm);
		\draw[fill = red] (9,-5) circle(0.1cm);
		\draw[fill = red] (10,-5) circle(0.1cm);
            \draw[fill = green] (6,-6) circle(0.1cm);
            \draw[fill = green] (7,-6) circle(0.1cm);
            \draw[fill = yellow] (8,-6) circle(0.1cm);
            \draw[fill = red] (9,-6) circle(0.1cm);
            \draw[fill = red] (10,-6) circle(0.1cm);
            \draw[fill = red] (11,-6) circle(0.1cm);
           
            \draw (6,-6) -- (11,-6);
		\draw (6,-6) -- (6,-1);
		\draw (6,-5) -- (7,-6);
		\draw (6,-4) -- (8,-6);
            \draw (6,-3) -- (9,-6);
            \draw (6,-2) -- (10,-6);
            \draw (6,-1) -- (11,-6);
            \draw (6,-2) -- (7,-2);
            \draw (6,-3) -- (8,-3);
            \draw (6,-4) -- (9,-4);
		\draw (6,-5) -- (10,-5);
		\draw (7,-6) -- (7,-2);
		\draw (8,-6) -- (8,-3);
            \draw (9,-6) -- (9,-4);
            \draw (10,-6) -- (10,-5);

            \draw[fill = green] (12,-1) circle(0.1cm);
		\draw[fill = green] (12,-2) circle(0.1cm);
		\draw[fill = green] (13,-2) circle(0.1cm);
		\draw[fill = green] (12,-3) circle(0.1cm);
		\draw[fill = green] (13,-3) circle(0.1cm);
            \draw[fill = green] (14,-3) circle(0.1cm);
            \draw[fill = green] (12,-4) circle(0.1cm);
            \draw[fill = green] (13,-4) circle(0.1cm);
            \draw[fill = yellow] (14,-4) circle(0.1cm);
            \draw[fill = yellow] (15,-4) circle(0.1cm);
            \draw[fill = green] (12,-5) circle(0.1cm);
		\draw[fill = green] (13,-5) circle(0.1cm);
		\draw[fill = yellow] (14,-5) circle(0.1cm);
		\draw[fill = yellow] (15,-5) circle(0.1cm);
		\draw[fill = red] (16,-5) circle(0.1cm);
            \draw[fill = green] (12,-6) circle(0.1cm);
            \draw[fill = green] (13,-6) circle(0.1cm);
            \draw[fill = yellow] (14,-6) circle(0.1cm);
            \draw[fill = red] (15,-6) circle(0.1cm);
            \draw[fill = red] (16,-6) circle(0.1cm);
            \draw[fill = red] (17,-6) circle(0.1cm);
           
            \draw (12,-6) -- (17,-6);
		\draw (12,-6) -- (12,-1);
		\draw (12,-5) -- (13,-6);
		\draw (12,-4) -- (14,-6);
            \draw (12,-3) -- (15,-6);
            \draw (12,-2) -- (16,-6);
            \draw (12,-1) -- (17,-6);
            \draw (12,-2) -- (13,-2);
            \draw (12,-3) -- (14,-3);
            \draw (12,-4) -- (15,-4);
		\draw (12,-5) -- (16,-5);
		\draw (13,-6) -- (13,-2);
		\draw (14,-6) -- (14,-3);
            \draw (15,-6) -- (15,-4);
            \draw (16,-6) -- (16,-5);

            \draw[fill = green] (0,-7) circle(0.1cm);
		\draw[fill = green] (0,-8) circle(0.1cm);
		\draw[fill = green] (1,-8) circle(0.1cm);
		\draw[fill = green] (0,-9) circle(0.1cm);
		\draw[fill = green] (1,-9) circle(0.1cm);
            \draw[fill = green] (2,-9) circle(0.1cm);
            \draw[fill = green] (0,-10) circle(0.1cm);
            \draw[fill = green] (1,-10) circle(0.1cm);
            \draw[fill = green] (2,-10) circle(0.1cm);
            \draw[fill = yellow] (3,-10) circle(0.1cm);
            \draw[fill = green] (0,-11) circle(0.1cm);
		\draw[fill = green] (1,-11) circle(0.1cm);
		\draw[fill = yellow] (2,-11) circle(0.1cm);
		\draw[fill = yellow] (3,-11) circle(0.1cm);
		\draw[fill = red] (4,-11) circle(0.1cm);
            \draw[fill = green] (0,-12) circle(0.1cm);
            \draw[fill = green] (1,-12) circle(0.1cm);
            \draw[fill = yellow] (2,-12) circle(0.1cm);
            \draw[fill = yellow] (3,-12) circle(0.1cm);
            \draw[fill = red] (4,-12) circle(0.1cm);
            \draw[fill = red] (5,-12) circle(0.1cm);
           
            \draw (0,-12) -- (5,-12);
		\draw (0,-12) -- (0,-7);
		\draw (0,-11) -- (1,-11);
		\draw (0,-10) -- (2,-12);
            \draw (0,-9) -- (3,-12);
            \draw (0,-8) -- (4,-12);
            \draw (0,-7) -- (5,-12);
            \draw (0,-8) -- (1,-8);
            \draw (0,-9) -- (2,-9);
            \draw (0,-10) -- (3,-10);
		\draw (0,-11) -- (4,-11);
		\draw (1,-12) -- (1,-8);
		\draw (2,-12) -- (2,-9);
            \draw (3,-12) -- (3,-10);
            \draw (4,-12) -- (4,-11);

            \draw[fill = green] (6,-7) circle(0.1cm);
		\draw[fill = green] (6,-8) circle(0.1cm);
		\draw[fill = green] (7,-8) circle(0.1cm);
		\draw[fill = green] (6,-9) circle(0.1cm);
		\draw[fill = green] (7,-9) circle(0.1cm);
            \draw[fill = green] (8,-9) circle(0.1cm);
            \draw[fill = green] (6,-10) circle(0.1cm);
            \draw[fill = green] (7,-10) circle(0.1cm);
            \draw[fill = green] (8,-10) circle(0.1cm);
            \draw[fill = yellow] (9,-10) circle(0.1cm);
            \draw[fill = green] (6,-11) circle(0.1cm);
		\draw[fill = green] (7,-11) circle(0.1cm);
		\draw[fill = green] (8,-11) circle(0.1cm);
		\draw[fill = yellow] (9,-11) circle(0.1cm);
		\draw[fill = yellow] (10,-11) circle(0.1cm);
            \draw[fill = green] (6,-12) circle(0.1cm);
            \draw[fill = green] (7,-12) circle(0.1cm);
            \draw[fill = green] (8,-12) circle(0.1cm);
            \draw[fill = yellow] (9,-12) circle(0.1cm);
            \draw[fill = red] (10,-12) circle(0.1cm);
            \draw[fill = red] (11,-12) circle(0.1cm);
           
            \draw (6,-12) -- (11,-12);
		\draw (6,-12) -- (6,-7);
		\draw (6,-11) -- (7,-11);
		\draw (6,-10) -- (8,-12);
            \draw (6,-9) -- (9,-12);
            \draw (6,-8) -- (10,-12);
            \draw (6,-7) -- (11,-12);
            \draw (6,-8) -- (7,-8);
            \draw (6,-9) -- (8,-9);
            \draw (6,-10) -- (9,-10);
		\draw (6,-11) -- (10,-11);
		\draw (7,-12) -- (7,-8);
		\draw (8,-12) -- (8,-9);
            \draw (9,-12) -- (9,-10);
            \draw (10,-12) -- (10,-11);

            \draw[fill = green] (12,-7) circle(0.1cm);
		\draw[fill = green] (12,-8) circle(0.1cm);
		\draw[fill = green] (13,-8) circle(0.1cm);
		\draw[fill = green] (12,-9) circle(0.1cm);
		\draw[fill = green] (13,-9) circle(0.1cm);
            \draw[fill = green] (14,-9) circle(0.1cm);
            \draw[fill = green] (12,-10) circle(0.1cm);
            \draw[fill = green] (13,-10) circle(0.1cm);
            \draw[fill = green] (14,-10) circle(0.1cm);
            \draw[fill = green] (15,-10) circle(0.1cm);
            \draw[fill = green] (12,-11) circle(0.1cm);
		\draw[fill = green] (13,-11) circle(0.1cm);
		\draw[fill = green] (14,-11) circle(0.1cm);
		\draw[fill = yellow] (15,-11) circle(0.1cm);
		\draw[fill = yellow] (16,-11) circle(0.1cm);
            \draw[fill = green] (12,-12) circle(0.1cm);
            \draw[fill = green] (13,-12) circle(0.1cm);
            \draw[fill = green] (14,-12) circle(0.1cm);
            \draw[fill = yellow] (15,-12) circle(0.1cm);
            \draw[fill = yellow] (16,-12) circle(0.1cm);
            \draw[fill = red] (17,-12) circle(0.1cm);
           
            \draw (12,-12) -- (17,-12);
		\draw (12,-12) -- (12,-7);
		\draw (12,-11) -- (13,-11);
		\draw (12,-10) -- (14,-12);
            \draw (12,-9) -- (15,-12);
            \draw (12,-8) -- (16,-12);
            \draw (12,-7) -- (17,-12);
            \draw (12,-8) -- (13,-8);
            \draw (12,-9) -- (14,-9);
            \draw (12,-10) -- (15,-10);
		\draw (12,-11) -- (16,-11);
		\draw (13,-12) -- (13,-8);
		\draw (14,-12) -- (14,-9);
            \draw (15,-12) -- (15,-10);
            \draw (16,-12) -- (16,-11);

            \draw[fill = green] (0,-13) circle(0.1cm);
		\draw[fill = green] (0,-14) circle(0.1cm);
		\draw[fill = green] (1,-14) circle(0.1cm);
		\draw[fill = green] (0,-15) circle(0.1cm);
		\draw[fill = green] (1,-15) circle(0.1cm);
            \draw[fill = green] (2,-15) circle(0.1cm);
            \draw[fill = green] (0,-16) circle(0.1cm);
            \draw[fill = green] (1,-16) circle(0.1cm);
            \draw[fill = green] (2,-16) circle(0.1cm);
            \draw[fill = green] (3,-16) circle(0.1cm);
            \draw[fill = green] (0,-17) circle(0.1cm);
		\draw[fill = green] (1,-17) circle(0.1cm);
		\draw[fill = green] (2,-17) circle(0.1cm);
		\draw[fill = green] (3,-17) circle(0.1cm);
		\draw[fill = yellow] (4,-17) circle(0.1cm);
            \draw[fill = green] (0,-18) circle(0.1cm);
            \draw[fill = green] (1,-18) circle(0.1cm);
            \draw[fill = green] (2,-18) circle(0.1cm);
            \draw[fill = green] (3,-18) circle(0.1cm);
            \draw[fill = yellow] (4,-18) circle(0.1cm);
            \draw[fill = yellow] (5,-18) circle(0.1cm);
           
            \draw (0,-18) -- (5,-18);
		\draw (0,-18) -- (0,-13);
		\draw (0,-17) -- (1,-17);
		\draw (0,-16) -- (2,-18);
            \draw (0,-15) -- (3,-18);
            \draw (0,-14) -- (4,-18);
            \draw (0,-13) -- (5,-18);
            \draw (0,-14) -- (1,-14);
            \draw (0,-15) -- (2,-15);
            \draw (0,-16) -- (3,-16);
		\draw (0,-17) -- (4,-17);
		\draw (1,-18) -- (1,-14);
		\draw (2,-18) -- (2,-15);
            \draw (3,-18) -- (3,-16);
            \draw (4,-18) -- (4,-17);
            
        \end{tikzpicture}
        }

        \caption{The third stage of fully clearing $T_5$.}
	\label{fig:thirdstep}
\end{figure}
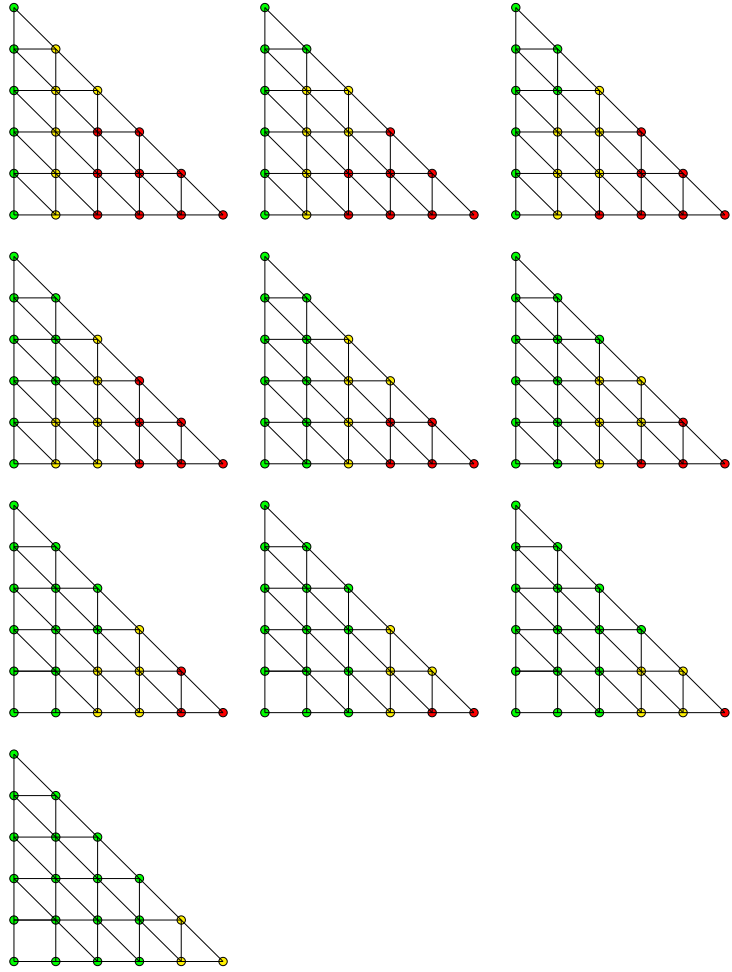

We now wish to show that 
$\In(T_n) > \frac{n}{\sqrt{2}}$. 
We will prove the lower bound by using the following lemma of Bernshteyn and Lee~\cite{Search}.

\begin{lemma}\label{prop3.1}
    Let $G$ be a connected graph. If $\In(G) \le m$, then for each $1 \le i \le |V(G)|$, there exists some $C \subseteq V(G)$ such that $i-m < |C| <i$ and 
    $|\partial{'}_{G}(C)| < m$
    where $\partial{'}_{G}(C)$ is the collection of all vertices in $C$ that share an edge with some vertex not in $C$.
\end{lemma}

We will adapt Lemma~\ref{prop3.1} into a new lemma in terms of the exterior boundary $\partial_{G}(D)$ instead of the interior boundary $\partial{'}_{G}(C)$. 

\begin{lemma}\label{adaptprop3.1}
    If $\In(G) \le m$, then for each $0 \le i \le |V(G)|-1$, there exists some 
    $D \subseteq V(G)$ such that $i < |D| < i+m$ and 
    $|\partial_{G}(D)| < m$.
\end{lemma}

\begin{proof}
Substituting $i$ with $|V(G)|-i$ in Lemma~\ref{prop3.1}, there exists $C \subseteq V(G)$ such that $|\partial{'}_{G}(C)|<m$ and $|V(G)|-i-m < |C| < |V(G)|-i$. Let $D = V(G) \setminus C$. 
Then, 
$|\partial_{G}(D)| = |\partial{'}_{G}(C)| < k$
and $i < |D| < i+k$ as desired.
\end{proof}

To show that $\In(T_n) > m$ where $m = \lfloor\frac{n}{\sqrt{2}}\rfloor$, it suffices by Lemma~\ref{adaptprop3.1} to show that there exists an $i$ such that every $D \subseteq V(T_n)$ with $i < |D| < i+m$ must satisfy $|\partial_{T_n}(D)| \ge m$. 

We shall show that $i = (m+1)+(m+2)+\dots+(n+1)$ works. Since $|D| < i+m = m+(m+1)+\dots(n+1)$, we conclude that the final segment of the simplicial ordering of size $|D|$ has boundary size at least $m$ by Lemma~\ref{initialandfinal}. Next, we will prove that the initial segment of the simplicial ordering of size $|D|$ has boundary size at least $m$. 
By Lemma~\ref{initialandfinal}, it is sufficient to show that $1+2+\dots+(m-1) < |D|$. 
From $|D| > i = (m+1)+(m+2)+\dots+(n+1)$, it suffices to show that 
\[(m+1)+(m+2)+\dots+(n+1) \ge 
1+2+\dots+(m-1).\]
This is equivalent to $(n+1)(n+2) \ge 2m^2$ which holds for $m = \lfloor\frac{n}{\sqrt{2}}\rfloor$.
Since both the initial and final segments of the simplicial ordering of size $|D|$ have boundary size at least $m$, we use Theorem~\ref{thm:inequality} to conclude that $|\partial_{T_n}(D)| \ge m$. Therefore, $\frac{n}{\sqrt{2}} < \In(T_n) \le \lceil\frac{3n}{4}\rceil +2$. 
\qedhere

% Assume for the sake of contradiction that $\In(T_n) \le \lfloor \frac{n}{\sqrt{2}} \rfloor = k$. 
% From Lemma~\ref{adaptprop3.1}, we get that for each 
% $0 \le i \le |V(T_n)|-1$, there exists $D \subseteq V(T_n)$ such that $|\partial_{T_n}(D)| < k$ and $i < |D| < i+k$. 
% Next, we will show that 
% \[(n+1)+n+\dots+(k+1) \ge 1+2+\dots+(k-1).\]
% Clearly, this inequality is equivalent to $n^2+3n+1 \ge 2k^2$. Thus, $k = \lfloor \frac{n}{\sqrt{2}} \rfloor$ satisfies the condition. Therefore, 
% \[(n+1)+n+\dots+(k+1) \ge 1+2+\dots+(k-1).\]
% Let us choose $i = (n+1)+n+\dots+(k+1)$ and 
% $i+k = (n+1)+n+\dots+(k+1)+k$. 
% Since
% \[(n+1)+n+\dots+(k+1) \ge 1+2+\dots+(k-1),\]
% we get that 
% \[1+2+\dots+(k-1) < |C|.\]
% From
% \[1+2+\dots+(k-1) < |C|\]
% and
% \[(n+1)+n+\dots+(k+1) < |C| <(n+1)+n+\dots+k,\]
% we conclude that both the initial segment of the simplicial ordering and the final segment of the simplicial ordering of size $|C|$ have boundary size of at least $k$ by using Lemma~\ref{initialandfinal}.
% Hence,  Theorem~\ref{thm:inequality} leads to the conclusion that $|\partial_{T_n}(C)| \ge k$ which is a contradiction. We get that $\In(T_n) > \lfloor \frac{n}{\sqrt{2}} \rfloor$. Since $\In(T_n)$ is always a whole number, we can also conclude that 
% $\frac{n}{\sqrt{2}} < \In(T_n) \le \left\lceil \frac{3n}{4} \right\rceil + 2$. 
% \qedhere

\end{proof}

%%%%%%%%%%%%%%%%%%%%
\section{Proof of Theorem~\ref{thm:lion}}
\label{sec:thm:lion}

We study the Lions and Contamination game on triangular grid graphs. This section is dedicated to proving Theorem~\ref{thm:lion}.

\begin{proof}[Proof of Theorem~\ref{thm:lion}]

First, we will give a proof for the upper bound. Note that Adams, Gibson, and Pfaffinger~\cite{Lions} have applied similar ideas to prove an upper bound for another similar graph. We will show that $l(T_n) \le n+1$ by demonstrating a sequence of movements with $n+1$ lions that clears the graph of contamination. Let the starting positions of all the lions be the $n+1$ vertices on column $0$ of the triangular grid graph. We will start off the lion movement by moving the first lion from $(0,0)$ to $(1,0)$. In the next turn, we shall move the second lion from $(0,1)$ to $(1,1)$. We will continue in this manner until the lion residing in $(0,n-1)$ moves to $(1,n-1)$, but we will leave the lion on vertex $(0,n)$ at the same vertex. Since the contamination cannot spread through an edge that the lions use to travel during a certain turn, we get that this sequence of moves will ensure that column $0$ and column $1$ are clear of contamination. Moreover, the previous movements also lead to all the vertices in column $1$ having a lion. We shall move the lions from column $1$ to column $2$ in a similar manner to the movement from column $0$ to column $1$. This will ensure that column $2$ becomes clear of contamination, while column $0$ and column $1$ also remain clear of contamination. We can continue this movement until all the columns of the triangular grid graph have become clear of contamination. Thus, $n+1$ lions are sufficient for clearing the triangular grid graph of contamination. Figure~\ref{fig:lionmovement} demonstrates the strategy mentioned above for $T_2$.

\begin{figure}\label{lionmove}
	\centering
	\begin{tikzpicture}
		\draw[fill = green] (0,0) circle(0.1cm);
		\draw[fill = green] (0,1) circle(0.1cm);
		\draw[fill = green] (0,2) circle(0.1cm);
		\draw[fill = red] (1,0) circle(0.1cm);
		\draw[fill = red] (1,1) circle(0.1cm);
		\draw[fill = red] (2,0) circle(0.1cm);
		
		\draw (0,0) -- (2,0);
		\draw (0,0) -- (0,2);
		\draw (1,0) -- (1,1);
		\draw (0,1) -- (1,1);
            \draw (0,1) -- (1,0);
            \draw (0,2) -- (2,0);

            \node at (-0.3,2) {L};
            \node at (-0.3,1) {L};
            \node at (-0.3,0) {L};

            \draw[fill = green] (3,0) circle(0.1cm);
		\draw[fill = green] (3,1) circle(0.1cm);
		\draw[fill = green] (3,2) circle(0.1cm);
		\draw[fill = green] (4,0) circle(0.1cm);
		\draw[fill = red] (4,1) circle(0.1cm);
		\draw[fill = red] (5,0) circle(0.1cm);
		
		\draw (3,0) -- (5,0);
		\draw (3,0) -- (3,2);
		\draw (4,0) -- (4,1);
		\draw (3,1) -- (4,1);
            \draw (3,1) -- (4,0);
            \draw (3,2) -- (5,0);

            \node at (2.7,2) {L};
            \node at (2.7,1) {L};
            \node at (4,-0.3) {L};

            \draw[fill = green] (6,0) circle(0.1cm);
		\draw[fill = green] (6,1) circle(0.1cm);
		\draw[fill = green] (6,2) circle(0.1cm);
		\draw[fill = green] (7,0) circle(0.1cm);
		\draw[fill = green] (7,1) circle(0.1cm);
		\draw[fill = red] (8,0) circle(0.1cm);
		
		\draw (6,0) -- (8,0);
		\draw (6,0) -- (6,2);
		\draw (7,0) -- (7,1);
		\draw (6,1) -- (7,1);
            \draw (6,1) -- (7,0);
            \draw (6,2) -- (8,0);

            \node at (5.7,2) {L};
            \node at (7,1.3) {L};
            \node at (7,-0.3) {L};

            \draw[fill = green] (9,0) circle(0.1cm);
		\draw[fill = green] (9,1) circle(0.1cm);
		\draw[fill = green] (9,2) circle(0.1cm);
		\draw[fill = green] (10,0) circle(0.1cm);
		\draw[fill = green] (10,1) circle(0.1cm);
		\draw[fill = green] (11,0) circle(0.1cm);
		
		\draw (9,0) -- (11,0);
		\draw (9,0) -- (9,2);
		\draw (10,0) -- (10,1);
		\draw (9,1) -- (10,1);
            \draw (9,1) -- (10,0);
            \draw (9,2) -- (11,0);

            \node at (8.7,2) {L};
            \node at (10,1.3) {L};
            \node at (11,-0.3) {L};
          
        \end{tikzpicture}

        \caption{Lion movement on $T_2$ that clears the graph of contamination. The red vertices are contaminated. The green vertices are cleared. The positions of the lions are marked by L}
	\label{fig:lionmovement}
\end{figure}
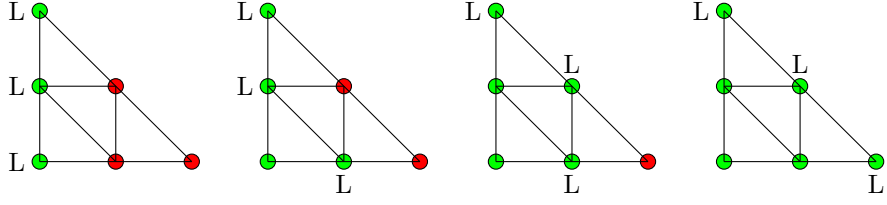

Next, we shall prove that $L(T_n) > \frac{n}{2\sqrt{2}}$. Since $\In(T_n) > \frac{n}{\sqrt{2}}$ by Theorem ~\ref{thm:search}, it is sufficient to show that $L(T_n) \ge \frac{\In({T_n})}{2}$. 
% We will start by proving the following claim.

% \begin{claim}\label{connect}
%     $L(T_n)>\frac{\In({T_n})}{2}$
% \end{claim}

% \begin{proof}

Consider the sequence of lion movements in the Lions and Contamination game that clears the triangular grid graph of contamination using $L(T_n)$ lions. Let $P(k)$ represent the set of positions of the lions after turn $k$ where $P(0)$ is the set of the initial positions of the lions. 
We desire to show that in the Zero-Visibility search game, if the searcher examines the vertices in $P(k-1) \cup P(k)$ in each turn the intruder will surely be caught. 
This implies that $\In({T_n}) \le 2L(T_n)$. Let us define $C_k$ as the set of vertices on the triangular grid graph that are clear of contamination after the lions move from $P(k-1)$ to $P(k)$. Similarly, define $I_k$ as the set of vertices on the triangular grid graph that have no possibility of containing the intruder after the searcher searches $P(k-1) \cup P(k)$. Next, we will prove the following claim

\begin{claim*}\label{claim}
$C_k \subseteq I_k$ for all $k$.
\end{claim*}

\begin{proof}
    We prove this claim using induction on $k$. Clearly, for $k = 0$ the claim is true since $C_0 = I_0 = P(0)$. Let 
    $k \ge 1$ be such that $C_{k-1} \subseteq I_{k-1}$. We shall prove that $C_k \subseteq I_k$.
    
    Let $x \in C_k$. Clearly, if 
    $x \in P(k-1) \cup P(k)$ then 
    $x \in I_k$. Suppose 
    $x \notin P(k-1) \cup P(k)$.
    By the definition of $C_k$, every neighbor of $x$ must be in $C_{k-1}$, and hence in 
    $I_{k-1}$ since 
    $C_{k-1} \subseteq I_{k-1}$. Therefore, $x \in I_k$ by definition of $I_k$.
    \qedhere
\end{proof}

Since the entire triangular grid graph will eventually be clear of contamination, we conclude that the intruder will eventually be caught by using the claim.
Moreover, 
$|P(k-1) \cup P(k)| \le |P(k-1)|+|P(k)| = 2L(T_n)$ 
for all $k$. 
Thus, $\In(T_n) \le 2L(T_n)$. This leads to the conclusion 
$\frac{n}{2\sqrt{2}} \le L(T_n) \le n+1$.
\qedhere
\end{proof}

%%%%%%%%%%%%%%%%%%%%
\section{Concluding Remarks}
\label{sec:conclusion}

% We have proved Theorem~\ref{thm:inequality} which was conjectured by Adams, Gibson and Pfaffinger~\cite{Lions}. Now we know that ice cream cone packings and row packings are best for minimizing the boundary size of a vertex subset of the triangular grid graph. Furthermore, we apply Theorem~\ref{thm:inequality} to the Zero Visibility k-Search Game and Lions and Contamination Game on triangular grid graphs. We use the theorem to determine the lower bound for both the Inspection Number and the Lion Number. We also demonstrate searching methods that surely work to give an upper bound for the Inspection Number and the Lion Number.
We would like to conclude the paper by proposing three problems that are related to the results of this paper. Bollob\'as and Leader~\cite{Compression} prove an isoperimetric inequality for rectangular grid graphs by introducing the compression technique. In Theorem~\ref{thm:inequality}, we prove an isoperimetric inequality for triangular grid graphs. We believe it would be interesting to know an isoperimetric inequality for other grid graphs. One might start with the next natural grid graph which is the hexagonal grid graph.

\begin{problem}\label{problem1}
Determine an isoperimetric inequality for any hexagonal grid graph.
\end{problem}

We consider the Zero Visibility $k$-Search Game and the Lions and Contamination Game on triangular grid graphs. We provide bounds for the inspection number and lion number, but we have not been able to determine a precise value for either function. Hence, the other two problems come as a natural consequence from our paper.

\begin{problem}
    For any triangular grid graph $T_n$, determine the value of $\In(T_n)$.
\end{problem}
   
\begin{problem}
     For any triangular grid graph $T_n$, determine the value of $L(T_n)$.
\end{problem}

\bibliographystyle{siam}
\bibliography{main}

\end{document}